\documentclass[a4paper, 12pt]{amsart}

\usepackage[T1]{fontenc}
\usepackage[utf8]{inputenc}

\usepackage{amssymb}
\usepackage{mathrsfs}
\usepackage{mathtools}
\usepackage{enumitem}
\usepackage{stmaryrd}
\usepackage{hyperref}
\usepackage{microtype}
\usepackage{xparse}
\usepackage{cite}
\usepackage{scalerel}
\usepackage{graphicx}
\usepackage{tikz}

\allowdisplaybreaks[2]

\newtheorem{theorem}{Theorem}[section]
\newtheorem{proposition}[theorem]{Proposition}
\newtheorem{lemma}[theorem]{Lemma}

\theoremstyle{definition}
\newtheorem{remark}[theorem]{Remark}
\newtheorem{definition}[theorem]{Definition}
\newtheorem{example}[theorem]{Example}

\numberwithin{equation}{section}

\theoremstyle{theorem}
\newtheorem{mainthm}{Theorem}

\theoremstyle{definition}

\newcommand{\N}{\mathbb N}

\newcommand{\Z}{\mathbb Z}

\newcommand{\I}{\mathbb I}

\newcommand{\Acal}{\mathcal{A}}
\newcommand{\Ical}{\mathcal{I}}

\newcommand{\AAA}{\mathfrak{A}}

\newcommand{\MMM}{\mathfrak m}
\newcommand{\NNN}{\mathfrak n}

\newcommand{\V}{\mathcal V}

\DeclareMathOperator{\M}{\mathcal{M}}

\DeclareMathOperator{\loc}{\Delta}
\DeclareMathOperator{\pp}{\mathfrak{p}}

\DeclareMathOperator{\Ext}{Ext}

\newcommand{\p}[1]{\partial #1}
\newcommand{\JH}[1]{[#1]}
\newcommand{\down}[1]{\downarrow #1}
\newcommand{\fracideal}[3]{{}_{#1} {#2} {}_{#3}}

\DeclareMathOperator{\ann}{ann}

\DeclareMathOperator{\diag}{diag}

\DeclareMathOperator{\Div}{Div}
\DeclareMathOperator{\modulo}{mod}
\DeclareMathOperator{\T}{\mathcal{T}}

\ExplSyntaxOn
\NewDocumentCommand{\cycle}{ O{\;} m }
{
	(
	\alec_cycle:nn { #1 } { #2 }
	)
}

\seq_new:N \l_alec_cycle_seq
\cs_new_protected:Npn \alec_cycle:nn #1 #2
{
	\seq_set_split:Nnn \l_alec_cycle_seq { , } { #2 }
	\seq_use:Nn \l_alec_cycle_seq { #1 }
}
\ExplSyntaxOff

\title[Multiplicative One-Sided Ideal Theory of HNP rings]{Multiplicative One-Sided Ideal Theory of Hereditary Noetherian Prime rings}

\author{Daniel Vitas}
\address{Department of Mathematics, Faculty of Mathematics and Physics,  University of Ljubljana, Slovenia \newline \indent 
Institute of Mathematics, Physics and Mechanics, Ljubljana, Slovenia}

\email{daniel.vitas@imfm.si}

%, 19A49}

\begin{document}

\begin{abstract}
	To any essential right ideal $I$ in a bounded HNP ring $R$ we may assign a divisor $\p{I}$, the image of the finite length module $R/I$ in the Grothendieck group $K_0(\text{fl.\ mod-$R$})$. We show that there is a composition of divisors $\circ$ for which $\p{I J} = \p{I} \circ \p{J}$. Additionally, we describe this composition, and show that $\p{}$ is faithful.
\end{abstract}

\maketitle

%%%%%%%%%%%%%%%%%%%%%%%%%%%%%%%%%%%%%%%%%%%%%%%
\section{Introduction}

Dedekind domains are one of the most important classes of rings. These are commutative integral domains that admit unique factorization of ideals, i.e., every nonzero ideal $I$ factors as
$$ I = P_1 \cdots P_n$$
for some maximal ideals $P_i$, unique up to a permutation.

In the noncommutative setting there are two natural generalizations of Dedekind domains. The first are Dedekind prime rings, that is, rings in which nonzero submodules of progenerators are progenerators \cite[\S 5.2.10]{McR}. One of the most important examples of Dedekind prime rings are classical maximal orders \cite[\S 8]{Re}.

Dedekind prime rings also admit unique factorization of two-sided ideals, i.e., every nonzero two-sided ideal is a product of maximal two-sided ideals, unique up to a permutation (see \cite[Theorem 5.2.9]{McR} or \cite[Theorem 22.10]{Re}). %Furthermore, two-sided ideals commute with each other.

However, in a noncommutative ring, understanding just two-sided ideals is often not enough, and we would like to say something about one-sided ideals as well. For this we need to consider a category of one-sided ideals (as a generalization of a monoid) with objects being equivalent Dedekind prime rings in a fixed simple artinian quotient ring $Q$, and morphisms between two Dedekind prime rings $R$ and $S$ being integral $(R,S)$-ideals, i.e., fractional $(R,S)$-ideals contained in $R \cap S$. The composition of two morphisms is, of course, the product of ideals. Dedekind prime rings also admit unique factorization in this category of one-sided ideals, namely, for any integral $(R,S)$-ideal $I$ we have
$$ \fracideal{R}{I}{S} =  \fracideal{R}{(P_1)}{T_1} \cdots  \fracideal{T_{n-1}}{(P_n)}{S} $$
for some integral $(T_{i-1},T_i)$-ideals $P_i$ (with $T_0 = R$ and $T_n = S$) that are maximal both as right $T_i$-ideals and as left $T_{i-1}$-ideals. The rings $T_i$ need not be unique, but nevertheless, the simple right $T_i$-modules $T_i/P_i$ are, using the canonical identification between simple modules of equivalent Dedekind prime rings $T_i$ and $S$, precisely the factors of the Jordan-Hölder decomposition of the finite length right $S$-module $S/I$; in this sense the factorization of $I$ is unique \cite[Theorem 2.14]{McR} (cf.\ \cite[Theorem 22.18, Theorem 22.24]{Re}).

The second generalization of Dedekind domains are hereditary noetherian prime (HNP) rings, that is, noetherian prime rings in which submodules of projective modules are projective. In particular, all Dedekind prime rings are also HNP rings. An example are classical hereditary orders \cite[Chapter 8]{Re} (see also Examples \ref{example-2} and \ref{ex-3}). HNP rings have a well developed module theory \cite{LR} and factorization theory \cite{Sm4}.

Multiplicative two-sided ideal theory in HNP rings is more complicated since there exist idempotent ideals, which is clearly not the case for Dedekind prime rings. Nevertheless, recently, Rump and Yang \cite{RY, Ru}, described multiplication of two-sided ideals using so called \emph{divisors}, i.e., formal integer combinations of maximal two-sided ideals. To a two-sided ideal $I$ in an HNP ring $R$ they assign a divisor $\p{I}$, a combination of maximal two-sided ideals corresponding to the composition series of the $R$-bimodule $R/I$. They show that there is an operation $\circ$ on the divisors that makes $\p{}$ a monoid embedding, and study this structure further.
%(Some other results regarding multiplicative two-sided ideal theory in HNP rings are given in \cite[\S 22]{LR} and\cite{Ru2, AM}.)

In this paper, we will develop the missing multiplicative one-sided ideal theory of HNP rings. We aim to generalize both the one-sided ideal theory of Dedekind prime rings and the two-sided ideal theory of HNP rings. Namely, we consider a class of bounded HNP rings, i.e., every essential one-sided ideal contains a nonzero two-sided ideal (e.g.\ every classical hereditary order is bounded). For an essential right ideal $I$ in an HNP ring $R$, the module $R/I$ has finite length \cite[Corollary 12.16]{LR}, and we can therefore consider its image in the Grothendieck group $K_0(\text{fl.\ $R$-mod})$, which is actually, by the Jordan-Hölder theorem, a free $\Z$-module with a basis of all simple $R$-modules; we denote this image by $\p{I}$ and call it a \emph{divisor} of $I$. (Note that, since in a bounded HNP ring simple modules are in a bijective correspondence with maximal two-sided ideals, this is a somewhat natural generalization of Rump and Yang's definition of a divisor.)

Since simple modules of two connected HNP rings $R$ and $S$ are not necessarily in a bijective correspondence (much less in a canonical one, as is the case for Dedekind prime rings), the most natural setting is to consider $\p{}$ as a functor from a category of one-sided ideals into a category of divisors with objects being HNP rings and morphisms from $R$ to $S$ being the set $K_0(\text{fl.\ $R$-mod})$; $\p{}$ fixes objects and to each one-sided ideal assigns its divisor. In this paper, we will show that there is precisely one composition of divisors for which $\p{}$ is a functor from the category of fractional ideals to the category of divisors.

\begin{mainthm} \label{uniquness-mainthm}
	There is precisely one composition of divisors
	$$\circ \colon K_0(\text{fl.\ $R$-mod}) \times K_0(\text{fl.\ $S$-mod}) \rightarrow K_0(\text{fl.\ $S$-mod})$$
	that makes bounded connected HNP rings and divisors between them a category such that $\p{}$ is a functor, the map
	$$ V_R \mapsto V_R \circ D_S - D_S$$
	is additive, and $V_R \circ 0_{R_0} = V_{R_0}$ for $R_0 \subseteq R$.
	%(By Theorem \ref{injectivity-mainthm}, the functor $\p{}$ is faithful.)
\end{mainthm}

Additionally, we will show that the functor $\p{}$ is faithful.

\begin{mainthm} \label{injectivity-mainthm}
	Let $R$ and $S$ be bounded HNP rings. Then,
	$$\p{} \colon \{\text{fractional $(R,S)$-ideals}\} \rightarrow K_0(\text{fl.\ $S$-mod})$$
	is injective.
\end{mainthm}

Furthermore, we will describe $\circ$ in Theorem \ref{general-mainthm}, and show that any divisor, even if it is not of the form $\p{I}$, is still determined by a collection of ideals in Theorem \ref{mult-mainthm}.

%We  show that there is precisely one composition of divisors for which $\p{}$ is a functor from the category of fractional ideals to the category of divisors (Theorem \ref{uniquness-mainthm}), and that $\p{}$ is faithful (Theorem \ref{injectivity-mainthm}). We further describe this operation (Theorems \ref{overring-mainthm} and \ref{general-mainthm}), and show that any divisor, even if it is not of the form $\p{I}$, is still determined by a collection of ideals (Theorem \ref{mult-mainthm}).
%
%Consider the following
%\begin{align*}
%	\text{objects} &= \{ T \subseteq Q \mid T \, \text{is a bounded HNP ring}\} \\
%	{\rm Hom} \, (S, R) &= \{ \fracideal{R}{I}{S} \subseteq Q \mid I \, \text{is a fractional (R,S)-ideal} \}.
%\end{align*}

%%%%%%%%%%%%%%%%%%%%%%%%%%%%%%%%%%%%%%%%%%%%%%%
\section{Preliminaries}

A (unital) ring $R$ is called an HNP ring if it is hereditary (i.e., both
left and right ideals are projective), noetherian (i.e.,
both left and right ideals are finitely generated), and prime (i.e., a product of nonzero ideals is nonzero). An HNP ring $R$ has a simple artinian ring of quotients $Q$. Throughout this paper, let $Q$ be fixed; any HNP ring $S$ considered is assumed to be a subring of $Q$ with $Q$ being its ring of quotients, i.e., every $q \in Q$ can be written as
$$ q = r s^{-1}$$
for some $r, s \in S$ with $s$ not being a zero-divisor.

An HNP ring $R$ is \emph{right (resp.\ left) bounded} if every essential right (resp.\ left) ideal contains a nonzero two-sided ideal. \cite[Proposition 23.4]{LR} shows that a ring is right bounded if and only if it is left bounded, so we will just write \emph{bounded}.
Every simple module in a bounded HNP ring $R$ is unfaithful, i.e., it has a nonzero annihilator (which is necessarily a maximal two-sided ideal by \cite[Lemma 12.14]{LR}). All classical hereditary orders are clearly bounded HNP rings.

We say that a simple module $V_R$ is the \emph{predecessor} of a simple module $W_R$ (and that $W$ is the \emph{successor} of $V$) if
$$ \Ext^1(V,W) \neq 0,$$
and denote
$$ V \rightarrow W.$$
If $R$ is bounded, for $\MMM = \ann V$ and $\NNN= \ann W$, by \cite[Proposition 15.7]{LR}, $V$ being the predecessor of $W$ is equivalent to
$$ (\MMM \cap \NNN) / \MMM \NNN \neq 0 .$$
Each simple module of a bounded HNP ring has a unique predecessor and a unique successor, and this relation splits simple modules of $R$ into (finite) cycles \cite[Proposition 23.4, Theorem 19.3]{LR}.
We will say that a cycle of simple $R$-modules is \emph{trivial} if it contains only a single simple module (which is then necessarily its own successor and predecessor).

We say that an overring $S \supseteq R$ \emph{kills} the simple $R$-module $V_R$ if
$$ V \otimes_R S = 0 .$$
By \cite[Theorem 13.12]{LR}, if $S$ does not kill $V$, then $V \otimes S$ is a simple $S$-module, and all simple $S$-modules are of this form.
Overrings of $R$ are precisely determined by the simple $R$-modules they kill, i.e., for a set of simple $R$-modules $\V$ there is a unique overring of $R$, denoted by $R(\V)$, that kills precisely the simple $R$-modules in $\V$, and no other simple $R$-module, and any overring of $R$ is equal to some $R(\V)$ \cite[Theorem 13.8]{LR}. We will say that an overring $S \supseteq R$ \emph{trivializes} a cycle of $R$-modules if it kills all but one module in that cycle; and say that it \emph{kills} a cycle if it kills all the modules in that cycle.

We say that another HNP ring $S$ is \emph{right (resp.\ left) finite} over $R$ if $S \supseteq R$ and $S/R$ is finitely generated as a right (resp.\ left) $R$-module.
By \cite[Corollary 27.10]{LR}, an overring is right finite if and only if it does not kill any cycle of simple $R$-modules.
Considering this together with \cite[Theorem 27.9, Corollary 27.10]{LR}, we see that an overring $S$ of a bounded HNP ring $R$ is right finite if and only if it is left finite (and we will just write \emph{finite} instead of right finite for bounded HNP rings).

Any finite subring $R$ of an HNP ring $S$ can be obtained from $S$ using the following construction; for a right ideal $A$ of the ring $S$, we define the \emph{idealizer} as
$$  R = \I_S(A) = \{ x\in S \mid xA \subseteq A\} .$$
We say that the ideal $A$ is \emph{isomaximal} (of type $U$) if $S/A \cong U^{(n)}$ for a simple $S$-module $U$ and $n \geq 1$, and we say that $A$ is \emph{generative} if $S A = A$.

\begin{proposition}
	Let $S$ be a bounded HNP ring and let $A$ be a right $S$-ideal. Then, $A$ is isomaximal and generative if and only if it is a proper right ideal and it strictly contains a maximal two-sided ideal.
\end{proposition}

\begin{proof}
If $A$ is isomaximal of type $U$, and $\MMM= \ann U$, then $(S/A) \, \MMM = 0$, i.e., $\MMM \subseteq A$, and if $A$ is in addition generative, then the inclusion must clearly be strict.

Conversely, if $\MMM \subset A$ for some maximal ideal $\MMM \lhd S$, then $A$ is clearly generative since $SA$ is a two-sided ideal strictly containing $\MMM$. The ideal $A$ is also isomaximal since $S/A$ is annihilated by $\MMM$, and therefore has natural $S/\MMM$-module structure. The ring $S/ \MMM$ is simple artinian by \cite[Lemma 12.14]{LR}, and therefore $S/A \cong U^{(n)}$ for the unique simple $S$-module $U$ annihilated by $\MMM$.
\end{proof}

If $A$ is an isomaximal generative ideal of the ring $S$, we call the ring $R = \I_S(A)$ a \emph{basic idealizer}. By \cite[Theorem 4.19]{LR}, if $S$ is an HNP ring, a basic idealizer $R$ is also an HNP ring, and the extension $S \supseteq R$ is finite. Of course, this is still true if we iterate this process (we take a basic idealizer from $R$, and so on); we call the ring obtained in this way an \emph{iterated basic idealizer}.

There is also a related notion; given chains of isomaximal generative ideals
$$ \Acal_i = \{ A_1 \subset A_2 \subset \dots \subset A_n\}$$
of distinct types $U_i$, and denoting $\Acal = \{\Acal_1, \ldots, \Acal_t \}$, the ring
$$\I_S(\Acal)  = \cap \{ \I_S(A) \mid A \in \Acal_i \,\, \text{for some $i$} \}$$
is called a \emph{multichain idealizer}.

\begin{remark} \label{remark}
	By \cite[Corollary 14.8]{LR}, $S$ is a finite overring of $R$ if and only if $R$ is an iterated  basic idealizer from $S$. Additionally, \cite[Theorem 8.6]{LR} shows that iterated basic idealizers are the same as multichain idealizers, so all of these three notions coincide.
\end{remark}

We say that a right $R$-module $\fracideal{}{I}{R} \subseteq Q$ is a \emph{right fractional $R$-ideal} if there is an $a \in Q$ that is not a zero-divisor such that
$$ a I \subseteq R ,$$
and $I$ contains some element that is not a zero-divisor. Denote by $\Ical(R)$ the set of all fractional $R$-ideals. Similarly, we define left fractional ideals, and say that an $(R,S)$-bimodule $\fracideal{R}{I}{S} \subseteq Q$ is a \emph{fractional $(R,S)$-ideal} if it is both a fractional right $S$-ideal and a fractional left $R$-ideal. Denote by $ \Ical_R(S)$ the set of all fractional $(R,S)$-ideals.
Note that since $R$ and $S$ are hereditary, the fractional ideal $\fracideal{R}{I}{S}$ is flat as both the left $R$-module and the right $S$-module.
This implies that for a right fractional $R$-ideal $\fracideal{}{J}{R}$ , we have an isomorphism of right $S$-modules
$$ J \otimes I \cong JI .$$
If $\fracideal{}{K}{R}$ is another right fractional $R$-ideal, we have
$$ (J \cap K) I = JI \cap KI ,$$
and if additionally $K \leq J$, then also
$$ J / K \otimes I \cong JI / KI .$$
Here, also note that by \cite[Lemma 12.14]{LR}, $J/K$ has finite length.
We will use all of these properties throughout this paper without further reference.

Given a fractional ideal $\fracideal{R}{I}{S}$, its \emph{left order} and \emph{right order} are defined as
$$ O_l(I) = \{ x \in Q \mid x I \subseteq I \} \quad \text{and} \quad O_r(I) = \{ x \in Q \mid Ix  \subseteq I \}.$$
The left (resp.\ right) order is a finite overring of $R$ (resp.\ $S$). We will also denote
\begin{align*}
	I^{-1} = \{ x \in Q \mid I x I \subseteq I \},
\end{align*}
which is a fractional $(S,R)$-ideal with
$$ I I^{-1} = O_l(I) \quad \text{and} \quad  I^{-1} I = O_r(I) .$$
We say that $I$ is \emph{invertible} if $O_l(I) = R$ and $O_r(I) = S$.

We say that HNP rings $R$ and $S$ are \emph{connected} if there is a fractional $(R,S)$-ideal.

\begin{proposition}
	Let $R$ and $S$ be connected HNP rings. Then, the ring $R$ is bounded if and only if $S$ is bounded. In particular, this is true for any left finite overring $S \supseteq R$.
\end{proposition}

\begin{proof}
	Assume that $S$ is bounded. Let $\fracideal{R}{I}{S}$ be a fractional $(R,S)$-ideal. First, we will show that $T= O_r(I)$ is bounded. Let $J_{T} \subseteq T$ be an essential ideal. The $S$-module
	$$ M_S = I/JI$$
	has finite length. Since $S$ is bounded, all simple $S$-modules have nonzero (and thus essential) annihilators. Therefore,
	$$ \AAA = \ann M_S$$
	is a nonzero two-sided ideal in $S$. If we multiply the inclusion
	$$ I \AAA \subseteq JI$$
	from the right by $I^{-1}$, we obtain
	$$ I \AAA I^{-1} \subseteq J T= J .$$
	Since $I \AAA I^{-1}$ is a nonzero two-sided ideal in $T$, this shows that $T$ is bounded. By \cite[Proposition 23.4]{LR}, $T$ has only cyclic towers, and therefore, since $T  \supseteq R$ is a finite overring, so does $R$ \cite[Theorem 20.2]{LR} (considering Remark \ref{remark}). Hence, $R$ is bounded. The converse follows by a similar argument.
\end{proof}

For fractional ideals $I \in \Ical_R(S)$ and $J \in \Ical_S(T)$, we have
$$ I J \in \Ical_R(T). $$
Indeed, we have elements $a, b, c, d \in Q$ that are not zero-divisors such that
$$ a I \subseteq S, \,\,  I b \subseteq R, \,\,  cJ \subseteq T, \,\, \text{and} \,\,  Jd \subseteq S.$$
Therefore,
$$ (c a) IJ \subseteq T \quad \text{and} \quad IJ (db) \subseteq R ,$$
and since $IJ$ also contains an element that is not a zero-divisor, this shows that it is an $(R,T)$-fractional ideal. This shows that we may consider a category with objects being HNP rings and morphisms fractional ideals between them. (This is in fact a small category, and should be understood as a generalization of a monoid.)

\begin{example} \label{example-2}
	Let $(\loc, \pp)$ be a local PID. The ring
	$$ R =  \I_{M_2(\loc)} \Big(\begin{bmatrix} \pp & \pp \\ \loc & \loc \end{bmatrix} \Big) = \begin{bmatrix} \loc & \pp \\ \loc & \loc \end{bmatrix} $$
	is an HNP ring, and in fact, a classical hereditary order. There are exactly two simple $R$-modules $V$ and $W$, successors to each other, with
	$$\rho(R, V) = \rho(R,W) = 1 .$$
	One can check that right fractional $R$-ideals are of the form
	\begin{align} \label{eq-ideal}
		a R, \,\, a M_2(\loc), \,\, \text{and} \,\,  a t M_2(\loc)t^{-1}
	\end{align}
	for an $a \in Q$ that is not a zero-divisor, and a matrix $t = \diag(\pi, 1)$, where $\pi$ generates $\pp$.
	
	Now, for an $a \in R$, consider a factorization in the category of one sided ideals
	\begin{align} \label{factorization}
		aR = I J ,
	\end{align}
	for some HNP ring $S$, and \emph{integral} ideals $\fracideal{aRa^{-1}}{I}{S}$ and $\fracideal{S}{J}{R}$, i.e., $I \subseteq aRa^{-1} \cap S$ and $J \subseteq S \cap R$. Since $J$ is a right fractional $R$-ideal, it has one of the three forms described in \eqref{eq-ideal}; however, since the right order of $aR$ is $R$, $J$ must be of the form $b R$ for some $b \in Q$. Therefore $S$ must be contained in $O_l(bR) = bRb^{-1}$, but the only such HNP ring (for which the inclusion is finite) is $bRb^{-1}$ itself. Similarly, we conclude that $I$ is of the form $b c R b^{-1}$ for some $c \in Q$. In fact, since we require both $I$ and $J$ to be integral, we have $b, c \in R$. Combining this with the equality \eqref{factorization}, we see that
	$$a = b c u$$
	for some unit $u \in R^\times$.
	Factoring $aR$ in the category of one-sided integral ideals is therefore equivalent to factoring $a$ in $R$.
%	
%	On the other hand, consider the matrix
%	$$ t = \begin{bmatrix} \pi & 0 \\ 0 & 1 \end{bmatrix} \in M_2(\loc).$$
%	Clearly, the matrix $t$ does not have a nontrivial factorization in $M_2(\loc)$, yet
%	$$ t M_2(\loc) = \big( t s R t^{-1}\big) \big( t M_2(\loc) \big)$$
%	for $s = \diag(1, \pi)$.
\end{example}

For an HNP ring $R$ with the set of all simple $R$-modules $\M_R$, we write
$$ \Div(R) = \bigoplus_{U \in \M_R} \Z \cdot U$$
for the formal integer combinations of simple modules. We call a
$$D = \sum_{U \in \M_R} k_U \cdot U \in \Div(R)$$
for some $k_U \in \Z$ (with only finitely many nonzero) a \emph{divisor}, and we will often write $D(U)$ for $k_U$.
%
%If $S$ is a right finite overring of $R$, then for every simple $R$-module $U_R$, the tensor product
%$$ U_{S} = U \otimes_R S$$
%is either zero or a simple $S$-module (LR Thm 13.12). For $D$ as above, it makes sense to write
%$$ D \otimes_R S = \sum_{\substack{U \in \M_R \\ U \otimes S \neq 0}} k_U \cdot U_S \text{.}$$

By the Jordan-Hölder theorem, the Grothendieck group of the category of finite length $R$-modules is isomorphic to $\Div(R)$, i.e.,
$$K_0(\text{fl.\ $R$-mod}) \cong \Div(R) \text{.}$$
For an arbitrary finite length module $M_R$, let $\JH{M} \in \Div(R)$ denote the image of $M$ via this isomorphism. If simple $R$-modules $U_1$, \ldots , $U_n$ appear in the Jordan-Hölder decomposition of $M$ with multiplicities $k_1$, \ldots , $k_n$, then
$$ \JH{M} = \sum_{i=1}^n k_i \cdot U_i \text{.}$$

For any fractional ideal $\fracideal{R}{I}{S}$,
we have
$$ \JH{M\otimes I} = \sum_{i=1}^n k_i \cdot \JH{U_i \otimes I},$$
and we will therefore write $ \JH{M} \otimes I$ for $\JH{M \otimes I}$, and note that this gives us an additive map from $\Div(R)$ to $\Div(S)$.
Indeed, this is because $\fracideal{R}{I}{}$ is flat, and a chain of modules
$$0 \subset M_1 \subset \dots \subset M_{n-1} \subset M_n = M$$
induces a chain
$$ 0 \subset M_1 \otimes I \subset \dots \subset M_{n-1} \otimes I \subset M_n \otimes I ,$$
and we have an isomorphism of $S$-modules
$$ (M_i \otimes I) / (M_{i-1} \otimes I) \cong ( M_i / M_{i-1}) \otimes I .$$

For $I_S \in \Ical(S)$ and any fractional ideal $J_S \subseteq I \cap S$, we set
$$ \p{I_S} = \JH{S/J} - \JH{I/J} \text{.}$$
This definition does not depend on the $J$ chosen. Indeed, let $K_S \subseteq I \cap S$ be another fractional ideal. By the Jordan-Hölder theorem applied to the diagram
\begin{figure}[!htbp]
	\begin{tikzpicture}[scale=0.7]
		\node (ul) at (-2,2) {$S$};
		\node (ur) at (2,2) {$I$};
		\node (m) at (0,0) {$J+K$};
		\node (bl) at (-2,-2) {$J$};
		\node (br) at (2,-2) {$K$};
		%	\node (zero) at (0,-2) {$0$};
		\draw (br) -- (m) -- (ur);
		\draw (bl) -- (m) -- (ul);
	\end{tikzpicture}
\end{figure}

\noindent we see that
$$ \JH{S/J} + \JH{I/K} = \JH{I/J} + \JH{S/K} \text{,}$$
which shows
$$ \p{I} = \JH{S/J} - \JH{I/J} = \JH{S/K} - \JH{I/K} \text{.}$$
Therefore, we have a well-defined map
$$ \p{} \colon \Ical(S) \rightarrow \Div(S) ,$$
which to each ideal assigns a divisor. Theorem \ref{injectivity-mainthm}, which we will prove in Section \ref{sec-inj}, shows that $\p{}$ restricted to $\Ical_R(S)$ is injective.

What we aim to define is an operation $\circ$ that generalizes ideal multiplication, i.e.,
$$ \p{J_R} \circ \p{I_S} = \p{JI_S}$$
for $J \in \Ical(R)$ and $I \in \Ical_R(S)$.
In other words, $\p{}$ is a functor between the category of fractional ideals (with natural multiplication) and the category of divisors (with the operation $\circ$) -- the objects in both categories are just HNP rings and are preserved by $\p{}$.

\begin{proposition}  \label{mult-tens}	
	For fractional ideals $J_R$ and $\fracideal{R}{I}{S}$, we have
	\begin{align*}
		\p{JI_S} = \p{J_R} \otimes I + \p{I_S} .
	\end{align*} 
\end{proposition}

\begin{proof}
Let $K = J \cap R$ and $L = S \cap  KI$. Then, by definition
$$ \p{J} = \JH{R/K} - \JH{J/K} ,$$
and by flatness of $I$, we have
\begin{align*}
\p{J} \otimes I &= \JH{R/K} \otimes I - \JH{J/K} \otimes I \\
&= \JH{I/KI} - \JH{JI/KI} \\
&= \JH{I/L} - \JH{JI/L}.
\end{align*}
On the other hand,
\begin{align*}
\p{JI_S} - \p{I_S} &= \JH{S/L} - \JH{JI/L} - (\JH{S/L} - \JH{I/L} ) \\
&= \JH{I/L} - \JH{JI/L} ,
\end{align*}
which concludes the proof.
\end{proof}

Considering Proposition \ref{mult-tens}, it makes sense to expect from $\circ$ that for any $D \in \Div(S)$, the map from $\Div(R)$ to $\Div(S)$, given by
$$ E \mapsto E \circ D - D \text{,}$$
is additive. Additionally, since $\otimes_R \cong \otimes_{R_0}$ for every finite subring $R_0 \subseteq R$, we also expect $$V_R \circ 0_{R_0} = V_{R_0} $$
for every $V \in \Div(R)$. Theorem \ref{uniquness-mainthm}, which we will prove in Section \ref{sec-gen} (existence) and Section \ref{sec-cat} (uniqueness), states that there is precisely one composition of divisors $\circ$ such that divisors form a category, $\p{}$ is a functor, and the above two requirements are satisfied.

%Motivate
%$V_R \circ 0_{R_0} = V_{R_0}$ for $R_0 \subseteq R$.

%In section \ref{sec-unique}, we will show that there is at most one operation $\circ$ satisfying this additivity property and functionality, described above. In section \ref{sec-overring}, we will define $\circ$ in the case where $S$ is a finite overring of $R$, and in section \ref{sec-gen}, we will give a general definition of $\circ$. Finally, in section \ref{sec-cat}, we will show that $\circ$ satisfies all the properties described above.

%\begin{mainthm} \label{uniquness-mainthm}
%	There is precisely one composition of divisors
%	$$\circ \colon \Div(R) \times \Div(S) \rightarrow \Div(S)$$
%	that makes bounded connected HNP rings and divisors between them into a category such that $\p{}$ is a functor, the map
%	$$ V_R \mapsto V_R \circ D_S - D_S$$
%	is additive, and $V_R \circ 0_{R_0} = V_{R_0}$ for $R_0 \subseteq R$.
%	%(By Theorem \ref{injectivity-mainthm}, the functor $\p{}$ is faithful.)
%\end{mainthm}

\begin{example} \label{ex-3}
	Let $(\loc, \pp)$ be a local PID. Similarly as in Example \ref{example-2}, the ring
	$$ R = \I_{M_3(\loc)}\Bigg(  \begin{bmatrix} \pp & \pp & \pp \\ \loc & \loc & \loc \\ \loc & \loc & \loc \end{bmatrix} \Bigg)=  \begin{bmatrix} \loc & \pp & \pp \\ \loc & \loc & \loc \\ \Delta & \loc & \loc \end{bmatrix} $$
	is an HNP ring (actually, a classical hereditary order) with two simple modules $V$ and $W$, successors to each other, with
	$$ I =  \ann V = \begin{bmatrix}
		\pp & \pp & \pp \\ \loc & \loc & \loc \\ \loc& \loc& \loc
	\end{bmatrix} \quad \text{and} \quad J =  \ann W =\begin{bmatrix}
		\loc & \pp & \pp \\ \loc & \pp & \pp \\ \loc& \pp & \pp
	\end{bmatrix} . $$
	It is easy to see that $\rho(R,V) = 1$ and $\rho(R,W) = 2$, and therefore
	$$ \p{ I} = V \quad \text{and} \quad \p{J} = 2 W.$$
	We have $J I = M_3(\pp)$, and we can see that
	$$ J/ (I \cap J) \cong V \quad \text{and} \quad (I \cap J )/ IJ \cong W ,$$
	and thus,
	$$V \circ 2W=  \p{IJ}= V + 3W.$$
\end{example}

	Maximal two-sided ideals in a bounded HNP ring $R$ are in bijective correspondence with simple modules, so at first glance, it seems reasonable to expect that, if we consider just two-sided ideals, the composition of divisors should yield exactly the theory developed by Rump and Yang \cite{RY}. However, the previous example shows that this is not the case. Since the divisor corresponding to the maximal ideal $J$ is $\p{J} = 2W$, it only makes sense to associate $2W$ to the Rump and Yang's divisor of $J$ (and $V$ to the Rump and Yang's divisor of $I$), but using this association, there is no obvious way to interpret the divisor $\p{IJ}= V + 3W$ (according to Rump and Yang's definition of a divisor, the divisor of $IJ$ should be the divisor of $I$ plus two times the divisor of $J$). The reason for this is that $\rho(R,V)$ and $\rho(R,W)$ differ; if $\rho(R,U)$ were constant on each cycle of simple modules (as is the case in Example \ref{example-2}), the two theories would coincide.

%%%%%%%%%%%%%%%%%%%%%%%%%%%%%%%%%%%%%%%%%%%%%%%%
\section{Injectivity of $\p{}$} \label{sec-inj}

We discussed the structure of simple modules of a bounded HNP ring, namely, that they form cycles with respect to the successor relation. If there is an invertible ideal connecting two HNP rings, it comes as no surprise that there is a bijection between simple modules (given by tensoring with this invertible ideal), which preserves the successor relation (and therefore, the cycle structure of simple modules).
%We will present the following lemma in slightly bigger generality than needed in this section as it will be useful in Section \ref{sec-gen}.

\begin{lemma} \label{inv-ideal}
	Let $R$ and $S$ be bounded HNP rings, and let $\fracideal{R}{I}{S}$ be an invertible fractional ideal. Then, for any simple $R$-module $U$, the module $U \otimes I$ is a simple $S$-module. Furthermore, if $V$ is the predecessor of $U$, then $V \otimes I$ is the predecessor of $U \otimes I$.
\end{lemma}

\begin{proof}
	Let $A_R \leq R$ to be a maximal one-sided ideal with
	\begin{align} \label{U-A} U_R \cong R/A \text{.}\end{align}
	Then, we have
	\begin{align} \label{U-A-I} U \otimes I = I / AI.\end{align}
	If $U \otimes I$ were not a simple $S$-module, then there would be a $B_S \leq S$ with
	\begin{align}\label{simple-inc}
		A I \subset B \subset I .
	\end{align}
	Tensoring (or multiplying) by $I^{-1}$ yields
	$$ A \subset B I^{-1} \subset R .$$
	By maximality of $A$, we have either $B I^{-1} = A$ or $B I^{-1} = R$. In either case, multiplying by $I$ contradicts \eqref{simple-inc}. Therefore, $U \otimes I$ is simple.
	
	Now, we will prove the second statement. Let $\MMM$ be the annihilator of $V$ and let $\NNN$ be the annihilator of $U$. By \cite[Proposition 15.7]{LR}, we have
	\begin{align} \label{m-cap-n} \MMM \cap \NNN \neq \MMM \NNN .\end{align}
	Since $I$ is flat, we have
	$$ I^{-1} (\MMM \cap \NNN) I = (I^{-1} \MMM I) \cap (I^{-1} \NNN I ) .$$
	Thus,
	\begin{align} \label{suc-conj}
		(I^{-1} \MMM I) \cap (I^{-1} \NNN I) \neq (I^{-1} \MMM I )(I^{-1} \NNN I ),
	\end{align}
	since if this were not the case, multiplying with $I$ from the left and $I^{-1}$ from the right would contradict \eqref{m-cap-n}.
	Since $\NNN$ is the annihilator of $U$, using \eqref{U-A}, we have
	$ \NNN \subseteq A$,
	and therefore,
	$$   \NNN I \subseteq A I .$$
	Now, by using \eqref{U-A-I}, we see that $I^{-1} \NNN I$ is the annihilator of $U \otimes I$ (since it is a maximal two-sided $S$-ideal). Similarly, $I^{-1} \MMM I$ is the annihilator of $V \otimes I$, and by \eqref{suc-conj} and \cite[Proposition 15.7]{LR}, $V \otimes I$ is the predecessor of $U \otimes I$.
\end{proof}

In a bounded HNP ring, for a minimal inclusion of two-sided ideals $I \subset J$, there is a unique maximal ideal $\MMM$ such that $J \MMM \subseteq I$ \cite[Proposition 7]{Ru}. The following lemma is a slight generalization of this fact.

\begin{lemma} \label{min-inc}
		Let $R$ and $S$ be bounded HNP rings, and let $I \subsetneq J$ be fractional $(R,S)$-ideals with the inclusion being minimal, i.e., there is no fractional $(R,S)$-ideal $K$ with $I \subset K \subset J$. Then,
		$$ (J/I)_S \cong U^{(n)}$$
		for some simple $S$-module $U$ and $n \in \N$.
\end{lemma}

\begin{proof}
	Let $\fracideal{}{L}{S} \subsetneq J$ be a maximal fractional right $S$-ideal that contains $I$, and denote by $\MMM$ the annihilator of the simple $S$-module $J/L$. Then, the fractional $(R,S)$-ideal $J \MMM + I$ is contained in $L$, and it contains $I$. Since the inclusion $I \subset J$ is minimal, we obtain $I = J\MMM + I $, or equivalently, $J \MMM\subseteq I$. This shows that the right $S$-module $J/I$ is annihilated by $\MMM$, and therefore has $S/\MMM$-module structure. Since $S/\MMM$ is a simple artinian ring \cite[Lemma 12.14]{LR}, the conclusion follows.
\end{proof}

We are now in position to show that $\p{}$ is injective. %and with that prove Theorem \ref{injectivity-mainthm}.

%\begin{theorem} \label{inj-overring}
%	Let $R$ and $S$ be bounded HNP rings. Then, $\p{}$ is injective on $\Ical_R(S)$.
%\end{theorem}

\begin{proof}[Proof of Theorem \ref{injectivity-mainthm}]
	Let $I$ and $J$ be two fractional $(R,S)$-ideals with $\p{I} = \p{J}$.
	For contradiction, assume that $I \subset I +J$, and let $K \subset I + J$ be a fractional $(R,S)$-ideal containing $I$ with the inclusion being minimal. By Lemma \ref{min-inc}, we have
	$$ (I+J)/K \cong U^{(n)}$$
	for some simple $S$-module $U$. Let $T$ be a finite overring of $S$ that kills all the simple modules in the cycle of $U$, except $U$ itself. By Proposition \ref{mult-tens}, we have
	\begin{align} \label{div-eq}
		\p{IT} = \p{JT}.
	\end{align}
	We have
	$$ (I+J)T/KT \cong (I+J)/K \otimes T \cong (U \otimes T)^{(n)}.$$
	Since $T$ does not kill $U$, the module $U \otimes T$ is a simple $T$-module; denote by $\MMM \lhd T$ its annihilator.
	%Indeed, if there were a fractional $(R,T)$-ideal $L$ with
%	$$ KT \subset L \subset (I+J) T,$$
%	then by multiplying with $T^{-1}$ and considering [ref], we would obtain
%	$$ K $$
%	Let $K \subsetneq (I + J)T$ be a fractional $(R,T)$-ideal containing $IT$ with the inclusion being minimal, and denote by $\MMM \lhd T$ the maximal ideal such that $(I+J)\MMM\subseteq K$, which exists by Lemma \ref{min-inc}.
	Let
	$$JT = L_0 \subset L_1 \subset \dots \subset L_{n-1} \subset  L_n = (I+J)T $$
	be a chain of fractional $(R,T)$-ideals with inclusions being minimal. By \eqref{div-eq}, the module $U \otimes T$ appears in the Jordan-Hölder decomposition of $L_{i}/L_{i-1}$ for some $i$. Lemma \ref{min-inc} then shows
	$$L_{i} / L_{i-1} \cong (U \otimes T)^{(m)} .$$
	On the other hand, since $(I+J)T \MMM \subseteq KT$, clearly $L_{i } \MMM \subseteq KT \cap L_{i}$, and therefore,
	$$ L_{i } / KT \cap L_{i} \cong (U \otimes T)^{(k)} .$$
	
	Consider the $(R, O_l(L_{i}))$-ideals $A = (KT \cap L_{i}) L_{i}^{-1}$ and $B = L_{i-1} L_{i}^{-1}$. We claim that $A$ and $B$ are generative, isomaximal right $O_l(L_{i})$-ideals of type $U \otimes L_{i}^{-1}$, yet they are not comparable by inclusion. Since $R$ is a finite HNP subring of $O_l(L_{i})$, it is a multichain idealizer from $O_l(L_{i})$ (Remark \ref{remark}), and clearly
	$$R \subseteq \I_{O_l(L_{i})}(A) \cap \I_{O_l(L_{i})}(B) .$$
	This contradicts \cite[Corollary 8.18]{LR}.
	
	To show the claim, first, note that
	$$O_l(L_{i}) / A \cong L_{i} / (KT \cap L_{i}) \otimes L_{i}^{-1} \cong (U \otimes L_{i}^{-1})^{(k)} ,$$
	and
	$$ O_l(L_{i}) / B \cong L_{i} / L_{i-1} \otimes L_{i}^{-1} \cong (U \otimes L_{i}^{-1})^{(m)} ,$$
	which shows that $A$ and $B$ are proper isomaximal ideals of type $U \otimes L_{i}^{-1}$. To see that $U \otimes L_{i}^{-1}$ is a simple $O_l(L_{i})$-module, write
	$$U \otimes L_{i}^{-1} =  \big( U \otimes O_r(L_{i}) \big) \otimes L_{i}^{-1} .$$
	The ring $O_r(L_{i})$ is a finite overring of $T$, which trivializes the cycle of $U$; therefore, $O_r(L_{i})$ cannot kill $U$, and $U \otimes O_r(L_{i})$ is a simple module. Since $L_{i}^{-1}$ is an invertible fractional $(O_r(L_{i}), O_l(L_{i}))$-ideal, Lemma \ref{inv-ideal} shows that $U \otimes L_{i}^{-1}$ is simple.
	
	Second, we will show that $A$ and $B$ are not comparable by inclusion. We have $KT + L_{i} = (I+J)T$. The isomorphism theorem then implies  $(I+J)T/ L_{i}  \cong K/(K \cap L_{i})$, and similarly $(I+J)T/ L_{i-1}  \cong KT /(KT \cap L_{i-1})$. Since $L_{i-1} \subset L_{i}$, we have $K \cap L_{i-1} \neq KT \cap L_{i} $, and therefore, $L_{i-1}$ does not contain $KT \cap L_{i}$.
	Since the inclusion $L_{i-1} \subset L_{i}$ is minimal, we have
	$$ KT \cap L_{i}+ L_{i-1} = L_{i}. $$
	Multiplying by $L_{i}^{-1}$, we obtain
	$$A + B = O_l(L_{i}).$$
	In particular, since $A$ and $B$ are both proper ideals, this shows that they are not comparable by inclusion.
	
	Last, the right ideals $A$ and $B$ are generative. Indeed, let $\NNN \lhd O_l(L_{i})$ be the annihilator of $U \otimes L_i^{-1}$. Since $\NNN$ annihilates $O_l(L_{i})/A$ and $O_l(L_{i})/B$, we have
	$$ \NNN \subseteq A \cap B .$$
	In particular, $\NNN \subset A$ and $\NNN \subset B$, since if any of the inclusions were not proper, $A$ and $B$ would be comparable by inclusion.  Since $\NNN$ is a maximal two-sided ideal, this shows that $A$ and $B$ are generative.
\end{proof}

\section{Overring Case} \label{sec-overring}

Before giving the general definition of the operation $\circ$ satisfying Theorem \ref{uniquness-mainthm} (we will do that in Section \ref{sec-gen}), we will define it just in the case of overrings. To avoid confusion, we will use $\star$ to denote it in this special case. But first, we will define an auxiliary map $\phi$.

\begin{definition} \label{def}
	Let $R$ be a bounded HNP ring, $S$ a finite overring of $R$, $U$ a simple $S$-module, and let
	$$ V_{n} \rightarrow \dots \rightarrow V_{1} \rightarrow V_0$$
	be the full cycle of simple $R$-modules with $V_0 \otimes_R S = U$.
	%	Denote
	%	\begin{align*}
		%		r_j &= \rho(R, V_{j})\\
		%		r &= r_0 + r_1 + \dots + r_n \text{.}
		%	\end{align*}
	%	For a $k \in \Z$, let $k' \in \{ 0, 1, \ldots, r - 1 \}$ be such that
	%	$$ k \equiv k' \,\, (\modulo \, \rho(S, U) ) \text{.}$$
	For a $k \in \Z$, define the additive map
	$$\phi(k \cdot U) \colon \Div(R) \rightarrow \Div(S)$$
	with
	$$ \phi(k \cdot U) (V_{j}) = \begin{cases} U & \sum_{i=0}^{j-1} \rho(R, V_{i}) \leq \overline{k}  < \sum_{i=0}^{j} \rho(R, V_{i}) \\
		0 & \text{otherwise} \end{cases} \text{,}$$
	where $\overline{k} \in \{ 0, 1, \ldots, \rho(S, U) - 1 \}$ is such that
	$$ k \equiv \overline{k} \,\, (\modulo \, \rho(S, U) ) \text{,}$$
	and $\phi(k \cdot U)(W) = 0$ for any other simple $R$-module $W$.
	For any
	$$D  = \sum_{U \in \M_S} k_U \cdot U \in \Div(S)$$
	extend the above definition with
	$$ \phi(D) = \sum_{U \in \M_S} \phi(k_U \cdot U) \text{.}$$
	Note that $\phi(0 \cdot U)$ is \textbf{not} the zero map.
\end{definition}

\begin{definition} \label{star-def}
	Let $S$ be a finite overring of $R$, $V \in \Div(R)$, and $D \in \Div(S)$. We denote
	$$V \star D = \phi(D)(V) + D \in \Div(S) \text{,}$$
	where $\phi$ is given in Definition \ref{def}.
\end{definition}

%The divisor $V \circ D$ can be explicitly described if $S \supseteq R$.

%\begin{mainthm} \label{overring-mainthm}
%	Let $R$ be a bounded HNP ring, $S$ a finite overring of $R$, $V \in \Div(R)$, and $D \in \Div(S)$. Then, we have
%	$$V \circ D = \phi(D)(V) + D \text{.}$$
%\end{mainthm}
%
%
%We will later define $\circ$ and show that it satisfies the conditions of Theorem \ref{uniquness-mainthm}; Theorem \ref{overring-mainthm} will then trivially hold by Definition \ref{star-def}.

In this section, we will show Lemma \ref{ideal-lemma}, which states that if the ring $S$ is big enough and the ring $R$ is small enough, then every divisor $D \in \Div(S)$ comes from a unique fractional ideal $\fracideal{R}{I}{S}$, i.e.,
$$ D = \p{I}.$$
For such a $D$ and $V \in \Div(R)$, we will prove Theorem \ref{star-mult}, which states that
$$ V \star D = V \otimes I + D . $$
This gives us a convenient way to compute $\star$, and we can use this to show that it satisfies the desired properties (i.e., being consistent with multiplication of ideals).

To apply this even if $S$ (resp.\ $R$) is not big (resp.\ small) enough, we will establish Lemma \ref{down-lemma} and Lemma \ref{tensor-lemma}, which show that $\star$ behaves well when transitioning from the ring $R$ to a smaller ring $R_0$ and from the ring $S$ to a bigger ring $\overline{S}$.
Then, we will prove Lemma \ref{module-lemma}, which addresses the existence of this small ring $R_0$ and shows that fractional $(R_0,S)$-ideals can be explicitly described.%, and computes the tensor product of a simple $R_0$-module with a fractional ideal.

\begin{lemma} \label{down-lemma}
	Let $S \supseteq R$ be a finite extension of bounded HNP rings, and let $V \in \Div(R)$ and $D \in \Div(S)$.  For a finite subring $R_0 \subseteq R$, we have
	$$ \phi(D)(V) = \phi(D)(V_{R_0}) .$$
\end{lemma}

\begin{proof}
	By additivity of the tensor product and $\phi(D)$, we may assume that $V$ is a simple $R$-module, and by Definition \ref{def} and additivity, we may also assume that $D = k \cdot U$ for a simple $S$-module $U$. Let
	$$V_n \rightarrow \dots \rightarrow V_1 \rightarrow V_0$$
	be the cycle of simple $R$-modules containing $V$.
	By \cite[Theorem 20.2 ]{LR} (considering Remark \ref{remark}), over $R_0$ this cycle becomes 
	$$ V^{k_n-1}_{n} \rightarrow \dots \rightarrow V^{0}_{n} \rightarrow \dots \rightarrow V^{k_0-1}_{0} \rightarrow \dots \rightarrow V^{0}_{0} ,$$
	where $V^{k_i-1}_{i}, \ldots, V^{0}_{i}$ is the JH decomposition of $(V_i)_{R_0}$, where $V^{0}_{0}\otimes R = V_0$.
	If $V_i \otimes S \neq U$ for any $i$, then clearly, $V^j_i \otimes S \neq U$ for any $i$ and $j$, and therefore,
	$$ \phi(D)(V) =  0 = \phi(D)(V_{R_0}) .$$
	So we only need to consider the case where $V_i \otimes S = U$ for some $i$. By relabeling $V_i$ if needed, we may assume that $V_0 \otimes S = U$.
	If we denote 
	$$ r_i = \rho(R, V_i) \quad \text{and} \quad r_{ij} = \rho(R_0, V^{j}_{i}) ,$$
	\cite[Theorem 32.19]{LR} shows
	\begin{align} \label{ri-sum}
		r_i = \sum_{j = 0}^{k_i-1} r_{ij} .
	\end{align}
	By Definition \ref{def} and equality \eqref{ri-sum}, we have
	$$ \phi(k \cdot U)(V^{j}_{i}) =\begin{cases} U & \sum_{l=0}^{i-1} r_{l} + \sum_{l=0}^{j-1} r_{il}  \leq  k' < \sum_{l=0}^{i-1} r_{l} + \sum_{l=0}^{j} r_{il} \\
		0 & \text{otherwise} \end{cases} .$$
	Therefore,
	\begin{align*}
		\phi(k \cdot U)((V_i)_{R_0}) &= \sum_{j=0}^{k_i-1} \phi(k \cdot U)(V^{j}_{i}) \\
		&= \begin{cases} U & \sum_{l=0}^{i-1} r_{l}  \leq  k' < \sum_{l=0}^{i} r_{l}   \\
			0 & \text{otherwise} \end{cases} ,
	\end{align*} 
	which concludes the proof.
\end{proof}

\begin{lemma} \label{tensor-lemma}
	Let $R$ be a bounded HNP ring, $S \supseteq R$ a finite overring, $V \in \Div(R)$ and $D \in \Div(S)$. For any finite overring $T \supseteq S$, we have
	$$\phi(D)(V) \otimes T= \phi(D \otimes T)(V) \text{.}$$
%	(The domain of the $\star$ operation on the right-hand-side is $\Div(R) \times \Div(\overline{S})$.)
\end{lemma}

\begin{proof}
%	Clearly, we only have to prove
%	$$\phi_S(D)(V) \otimes \overline{S} = \phi_{\overline{S}}(D \otimes \overline{S})(V) \text{.}$$
	Let $W$ be any simple $T$-module, and let $U$ be the simple $S$-module with $U \otimes T= W$.
	By Definition \ref{def}, it is clear that the coefficient of $\phi(D \otimes T )(V)$ at $W$ is the same as the coefficient of $\phi(D)(V) $ at $U$,
	which is obviously the same as the coefficient of $\phi(D)(V) \otimes T$ at $W$.
\end{proof}

\begin{lemma} \label{extensions}
	Let $R$ be an HNP ring and let $D, E \in \Div(R)$. If for every finite overring $S \supseteq R$ that trivializes the cycles of all simple modules appearing in $D$ and $E$ there is a finite overring $T \supseteq S$ such that
	$$ D \otimes T = E \otimes T\text{,}$$
	then $D = E$.
\end{lemma}

\begin{proof}
	Assume that there is a simple $R$-module $V$ such that the coefficients of the divisors $D$ and $E$ at $V$ differ.
	Let $S$ be a right finite overring of $R$ that trivializes the cycles of all simple modules appearing in $D$ and $E$, but doesn't kill $V$. Then, any finite overring $T\supseteq S$ cannot kill $V$ (as then it would kill a whole cycle, which would contradict \cite[Corollary 27.10]{LR}), therefore $V  \otimes T$ is a simple $T$-module, and since the coefficient of $D \otimes T$ at $V \otimes T$ is clearly the same as coefficient of $D$ at $V$, and similarly for $E$, we obtain
	$$ D \otimes T \neq E \otimes T \text{.} \eqno{\qedhere}$$
\end{proof}

\begin{proposition} \label{module-lemma}
	Let $S$ be a bounded HNP ring and
	$$ R = \I_S(\Acal)$$
	for $\Acal = \{ \Acal_1, \ldots, \Acal_t\}$ a multichain of type $U_1, \ldots, U_t$ with
	$$ \Acal_i = \{ A_{in_i} \subset \dots \subset A_{i1}\} \text{.}$$
	By \cite[Proposition 9.4]{LR}, $S$-module $U_i$ viewed as an $R$-module has the Jordan-Hölder decomposition
	$$ V_{i n_i} \rightarrow \dots \rightarrow V_{i1}  \rightarrow V_{i0}$$
	with $V_{i0} \otimes S = U_i$.
	Then, we have
	$$ V_{ij} \otimes_R (A_{1k_1} \cap \ldots \cap A_{t k_t}) = V_{ij} \otimes_R A_{ik_i} = \delta_{jk_i} U_i \text{,}$$
	where $0 \leq k_l \leq n_l$ (and $A_{l0} = S$).
\end{proposition}

\begin{proof}
	We proceed by induction on $n_i$. First, assume that $n_i = 0$, i.e., $\Acal_i  = \emptyset$, or rather, $\Acal_i$ does not even appear in $\Acal$. By \cite[Theorem 8.8]{LR}, the simple $S$-module $U_i$ is then also simple as an $R$-module (which we denote $V_{i0}$). Denote $\MMM = \ann U_i$.
	Since ideals $A_{lk}$ are generative, we have $S A_{lk} = S$ for every $l$, and therefore,
	$$ S(A_{1k_1} \cap \ldots \cap A_{t k_t}) = SA_{1k_1} \cap \ldots \cap SA_{t k_t} = S$$
	by flatness. The ideal $\fracideal{R}{B}{S} = A_{1k_1} \cap \ldots \cap A_{t k_t}$ is thus also generative, i.e., $S B =S$. We have
	$$S/\MMM \otimes B = S B / \MMM B  = S/\MMM ,$$
	where we used flatness of $B$ and equality $\MMM = \MMM S$. From $S/\MMM = U^{(n)}$, we conclude
	$$ V_{i0} \otimes B = U ,$$
	which is exactly the wanted conclusion since necessarily $k_i = 0$ and $A_{i0} = S$.
	
	Now assume that $n_i > 0$ and that the conclusion of the lemma holds for all smaller integers. By \cite[Theorem 8.4]{LR}, we have, for $T = \I_S(A_{in_i})$,
	$$ R = \I_T(\Acal_{\down{T}}) \text{,}$$
	where $\Acal_{\down{T}}$ consists precisely of chains $\Acal_l$ from $\Acal$, except that $A_{in_i}$ is omitted and $A_{lk_l}$ are replaced by
	$$ (A_{lk_l})_{\down{T}} = A_{lk_l} \cap T \text{,}$$
	considering \cite[Definition 7.1]{LR} Cases $(\alpha_1)$ and $(\alpha_2)$. By \cite[Theorem 4.4, Corollary 4.8]{LR}, $U_T$ has the Jordan-Hölder decomposition
	$$V \rightarrow W .$$
	By \cite[Theorem 8.8]{LR}, the module $V_R$ is simple, since $\Acal_{\down{T}}$ does not contain a (non-empty) chain of type $V$. Therefore, $V_R = V_{i n_i}$ and $W_R$ has Jordan-Hölder decomposition
	$$ V_{i,n_i-1} \rightarrow \dots \rightarrow V_{i0} \text{.}$$

	First consider the case where $j = n_i$. As before, denote $B = A_{1k_1} \cap \ldots \cap A_{t k_t}$, and let  $B_{\down{T}} = (A_{1k_1})_{\down{T}} \cap \ldots \cap (A_{t k_t})_{\down{T}}$.
	By the proof of \cite[Theorem 7.2]{LR},
	$$T (A_{lk})_{\down{T}} = \begin{cases} A_{i n_i} & l = i, \, k = n_i \\ T & \text{otherwise} \end{cases} ,$$
	and therefore, by flatness,
	$$ T B_{\down{T}}  = \begin{cases} A_{i n_i} & k_i = n_i \\ T & k_i \neq n_i \end{cases} .$$
	Thus,
	$$ S/T \otimes B_{\down{T}} = S B_{\down{T}} / T B_{\down{T}} = \begin{cases} S/A_{i n_i} & k_i = n_i \\ S/T & k_i \neq n_i \end{cases} .$$
	By \cite[Corollary 13.4 (ii)]{LR}, we have $(A_{lk})_{\down{T}}  \otimes_T S = A_{lk}$ for every $l$ and $k$, and therefore, by flatness,
	$$ B_{\down{T}} \otimes S = B ,$$
	and hence, we have
	$$ S/T \otimes B = S/T \otimes B_{\down{T}}  \otimes S = \begin{cases} S/A_{i n_i} & k_i = n_i \\ 0 & k_i \neq n_i \end{cases} .$$
	By \cite[Theorem 4.4]{LR},
	$$ S/T \cong V^{(n)}  \quad \text{and} \quad S/ A_{i n_i} = U_i^{(n)}\text{,}$$
	and by the above, we conclude that
	$$ V_{in_i} \otimes B = \delta_{n_i k_i}U_i .$$
	Now let $j < n_i$.
	If $k_i < n_i$, by induction hypothesis, we have
	$$ V_{ij} \otimes B_{\down{T}} = \delta_{j k_i} W .$$
	Tensoring by $S$ and using \cite[Proposition 4.7]{LR}, yields
	$$ V_{ij} \otimes B = \delta_{j k_i} U_i .$$
	On the other hand, if $k_i = n_i$, then by \cite[Theorem 4.4]{LR}, we have
	$$ T/A_{in_i} = W^{(n)} ,$$
	and tensoring by $B$ and using $TB = A_{i n_i}$ and $A_{i n_i}B = A_{i n_i} $, yields
	$$ T/A_{in_i} \otimes B = TB/A_{i ni} B = 0 .$$
	Therefore, in this case
	$$ V_{ij} \otimes B = 0 ,$$
	which concludes the proof.
\end{proof}

\begin{lemma} \label{ideal-lemma}
	Let $S$ be a bounded HNP ring and
	$$ R = \I_S(\Acal)$$
	for $\Acal = \{ \Acal_1, \ldots, \Acal_t\}$ a multichain of type $U_1, \ldots, U_t$ with
	$$ \Acal_i = \{ A_{in_i} \subset \dots \subset A_{i1}\} \text{.}$$
	Assume that for $i=1,\ldots, s$, the annihilators $\MMM_i = \ann U_i$ are invertible and the inclusions in the chain $\Acal_i$ are minimal  (including $A_{i n_i}$ being a minimal right ideal over $\MMM_i$, and $A_{i1}$ being a maximal ideal), and denote $A_{i0} = S$. Then, for the ideal
	$$ I_S = ( A_{1k_1} \cap  \dots \cap A_{sk_s} ) \,  \MMM^{l_1}_1 \dots \MMM^{l_s}_s \text{,}$$
	where $l_i \in \Z$ and $0 \leq k_i \leq n_i$ (and $A_{i0} = S$), we have
	$$ \p{I} = (k_1 +(n_1+1) l_1) \, U_1 + \dots + (k_s + (n_s + 1)l_s) \, U_s \in \Div(S) .$$
	In particular, for any divisor
	$$D \in \bigoplus_{i=1}^s \Z \cdot U_i \leq \Div(S)$$
	there is a unique fractional ideal $\fracideal{R}{I}{S}$ with $ D = \p{I}$.
\end{lemma}

\begin{proof}
	For $B = A_{1k_1} \cap  \dots \cap A_{sk_s}$, we have
	$$  \p{I} = \JH{S/I} =  \p{B} + \JH{B/B\MMM^{l_1}_1 \dots \MMM^{l_s}_s} .$$
	On one hand, $B \subseteq A_{ik_i}$, and therefore,
	$$ \p{B} \geq \p{A_{i k_i}} = k_i  U_i$$
	for every $i = 1,\ldots, s$. On the other hand,
	$$B \supseteq A_{ik_i} \MMM_1 \dots \MMM_{i-1} \MMM_{i+1} \dots \MMM_s ,$$
	and therefore,
	\begin{align*}
		\p{B} &\leq \p{(A_{ik_i} \MMM_1 \dots \MMM_{i-1} \MMM_{i+1} \dots \MMM_s)}\\
		&= l_1  U_1 + \dots + l_{i-1}  U_{i-1} + k_i U_i  + l_{i+1}  U_{i+1} + \dots + l_s  U_s
	\end{align*}
	for $i =1, \ldots, s$.
	Combining the two inequalities yields
	$$ \p{B} = k_1 \cdot U_1 + \dots + k_s U_s .$$
	Since $\MMM_i$ are invertible, by \cite[Lemma 15.3]{LR} and a straightforward induction, we obtain
	\begin{align*}
		\JH{B/B\MMM^{l_1}_1 \dots \MMM^{l_s}_s} &= \JH{S/\MMM^{l_1}_1 \dots \MMM^{l_s}_s} \\
		&= l_1 \p{\MMM_1} + \dots + l_s \p{\MMM_s} \\
		&= l_1 (n_1 + 1)  U_1 + \dots + l_s (n_s + 1) U_s .
	\end{align*}
	This concludes the proof.
\end{proof}

The following (together with Proposition \ref{ideal-lemma}) is a special case of Theorem \ref{mult-mainthm} for the operation $\star$.

\begin{theorem} \label{star-mult}
	Let $R_0$, $R$, and $S$ be bounded HNP rings such that $S \supseteq R \supseteq R_0$ is a chain of finite overrings. Let $V \in \Div(R)$, and $\fracideal{R_0}{I}{S}$ be a fractional ideal. Then, we have
	$$ V \star \p{I} = V \otimes I + \p{I} \text{.}$$
\end{theorem}

\begin{proof}
	Since both $\phi(\p{I})$ and $\otimes I$ are additive, and considering Lemma \ref{down-lemma}, we may assume that $V$ is a simple $R_0$-module. For simplicity, let us write $R$ for $R_0$ (as the ring $R$ does not play a role anymore).
	
	First, consider the case where $\p{I_S}$ and $\JH{V \otimes_{R} I}$ involves only modules $U_1, \ldots, U_s$ with trivial cycles. Let
	$$ R= \I_S(\Acal)$$
	for $\Acal = \{ \Acal_1, \ldots, \Acal_t\}$ a multichain of type $U_1, \ldots, U_t$, $t\geq s$, with
	$$ \Acal_i = \{ A_{in_i} \subset \dots \subset A_{i1}\} \text{.}$$
	For $i = 1, \ldots, s$, let $\Acal'_i$ be any finest chain containing $\Acal_i$, say
	\begin{align*}A^{r_{in_i}-1}_{in_i} \subset \dots \subset A^{0}_{in_i} \subset
		\dots \subset A^{r_{i1} -1}_{i1} \subset \dots  \subset A^0_{i,n_i-1} \subset  A^{r_{i0}-1}_{i0} \subset \dots \subset A^{1}_{i0} ,\end{align*}
	where $A^0_{ij} = A_{ij}$, the right $S$-ideal $A^{r_{in_i}-1}_{in_i}$ is minimal over $\ann U_i$ and $A^{1}_{i0}$ is a maximal ideal. By \cite[Proposition 9.4]{LR}, we have
	$$ r_{ij} = \rho(R, V_{ij}) .$$
	Consider the right finite subring
	$$ R_0 = \I_S(\Acal'_1, \ldots, \Acal'_s, \Acal_{s+1}, \ldots, \Acal_t) \subseteq R$$
	and label by
	$$ V^{r_{in_i}-1}_{in_i} \rightarrow \ldots \rightarrow V^{1}_{i0} \rightarrow V^{0}_{i0}$$
	the Jordan-Hölder decomposition of $U_i$ over $R_0$ (\cite[Theorem 20.2]{LR} and \cite[Proposition 9.4]{LR}).
	Since $I$ is an $(R,S)$-fractional ideal, it is also an $(R_0, S)$-fractional ideal, and by Lemma \ref{ideal-lemma}, we have
	$$ \fracideal{R_0}{I}{S} =  A^{m_1}_{1k_1} \cap  \dots \cap A^{m_s}_{sk_s} \, \MMM^{l_1}_1 \dots \MMM^{l_s}_s$$
	for some $l_j, k_j,m_j$ such that
	\begin{align} \label{eq-m}
		\p{I}(U_i) = m_i + \sum_{j=0}^{k_i-1} r_{ij} +  r_i l_i . 
	\end{align}
	%Since tensoring over $R$ is the same as tensoring over $R_0$ and,
	By \cite[Theorem 20.2]{LR}, we have
	\begin{align} \label{vij-eq}
		[V_{ij}]_{R_0} = \sum_{l=0}^{r_{ij}-1} V^{l}_{ij} ,
	\end{align}
	and therefore,
	$$V_{ij} \otimes I = \sum_{l=0}^{r_{ij}-1} V^{l}_{ij} \otimes ( A^{m_1}_{1k_1} \cap  \dots \cap A^{m_1}_{sk_s}) \,  \MMM^{l_1}_1 \dots \MMM^{l_s}_s .$$
	By Lemma \ref{module-lemma}, we have
	$$ V^{l}_{ij} \otimes ( A^{m_1}_{1k_1} \cap  \dots \cap A^{m_1}_{sk_s} ) =  \begin{cases} U_i & l = m_i, \,\, j = k_i \\ 0 & \text{otherwise.} \end{cases} $$
	For an invertible ideal $\MMM \lhd S$, by Lemma \ref{inv-ideal}, the $S$-module $U_i \otimes \MMM$ is simple, and since $\MMM_i = \ann U_i$ is invertible by assumption and invertible ideals in an HNP ring commute \cite[Theorem 22.15]{LR}, the maximal ideal $\MMM_i$ annihilates $U_i \otimes \MMM$, and therefore, $U_i \otimes \MMM \cong U_i$. Hence, we have
	$$ V_{ij} \otimes I  = \begin{cases}  U_i & j = k_i \\ 0 & \text{otherwise.} \end{cases}  $$
	Considering \eqref{eq-m}, $j = k_i$ is equivalent to
	$$  \sum_{l=0}^{j-1} r_{il} \leq k'<  \sum_{l=0}^{j} r_{il} ,$$
	where $k' \in \{0, 1, \ldots, r-1\}$ is such that
	$$  \p{I}(U_i) = k'  \,\, (\modulo \, r) .$$
	This proves the theorem in this special case.
	
	Now, consider an arbitrary case, i.e., $S$ is any right finite overring of $R$. Let $T$ be an overring of $S$ that trivializes the cycles of all simple modules appearing in $\phi(\p{I})(V)$ and $V \otimes I$. Clearly, we have
	$$ (V \otimes I) \otimes T= V \otimes I T  \text{,}$$
	and by Lemma \ref{tensor-lemma}, also
	$$\phi(\p{I_S})(V) \otimes T = \phi(\p{I} \otimes T)(V) =  \phi(\p{I T})(V) .$$
	Therefore, by the just proven special case (for the ring $T$ and the ideal $IT$), we have
	$$ \phi(\p{I})(V) \otimes T =( V \otimes I ) \otimes T .$$
	Applying Lemma \ref{extensions} concludes this proof.
\end{proof}

\section{General Case} \label{sec-gen}

%\begin{definition} \label{admissible-ring}
%	Let $R$ be an HNP ring and $U_1, \ldots, U_n$ simple $R$-modules contained in different cycles.
%	We say that an HNP ring $S$ is \emph{admissible} with respect to $U_1, \ldots, U_n$ if there is a fractional ideal $\fracideal{R}{K}{S}$ (which we will call an \emph{admissible ideal}) with $O_r(K) = S$ such that $\p{K}$ consists only of simple modules with trivial cycles and $U_i \otimes K$ is a simple $S$-module with a trivial cycle for any $i = 1,\ldots,n$.
%\end{definition}
%
%\begin{remark}
%	If $R$ has only finitely many non-trivial cycles of simple modules and for every such cycle there is a module among $U_1, \ldots, U_n$ contained in it, then an overring $S \supseteq R$ is admissible w.r.t.\ $U_1, \ldots, U_n$ if and only if it is a Dedekind prime ring. However, if $R$ has infinitely many non-trivial cycles of simple modules, there are no right finite overrings that are also Dedekind prime rings. In that case, the notion of an admissible overring is a suitable replacement. This definition also addresses rings $S$ that are not overrings of $R$.
%\end{remark}
%
%\begin{remark}
%	An overring $\overline{R} \supseteq R$ is admissible w.r.t. $U_1, \ldots, U_n$ in the sense of Definition \ref{admissible-overring} if and only if it is admissible w.r.t. $U_1, \ldots, U_n$ in the sense of Definition \ref{admissible-ring} (take $K = \overline{R}$).
%\end{remark}

This section aims to extend the results of the previous section to arbitrary two connected bounded HNP rings $R$ and $S$, not necessarily comparable by inclusion. For $V \in \Div(R)$ and $D \in \Div(S)$, the definition of $V \circ D$ is somewhat technical and is given later in this section (Definition \ref{circ-def}). However, we can use the following theorem to determine $V \circ D$; using a fractional ideal $\fracideal{R}{K}{S}$, which enables us to compare structures of simple modules between $R$ and $S$, we can obtain information about the tensor product $(V \circ D) \otimes O_r(K)$. To completely determine $V \circ D$, one must consider several different fractional $(R,S)$-ideals.

\begin{theorem} \label{general-mainthm}
	Let $R$ and $S$ be bounded HNP rings, and let $\fracideal{R}{K}{S}$ be a fractional ideal. For $V \in \Div(R)$ and $D \in \Div(S)$, we have
	$$ (V \circ D) \otimes {O_r(K)} = \phi(D \otimes K^{-1} + \p{K^{-1}}) (V) \otimes K + D \otimes O_r(K) .$$
\end{theorem}

%\begin{mainexample}
%	We say that an HNP ring $R$ is a \emph{Dedekind prime ring} if every fractional right ideal is invertible. The cycle structure of simple modules of a bounded Dedekind prime ring is completely trivial, i.e., every cycle consists only of a single simple module, which is its own successor. (This corresponds to the fact that Dedekind prime rings do not have finite extensions.) Therefore, it is clear that in Definition \ref{def} where $R$ is a Dedekind prime ring (and thus, necessarily, $S = R$), we have, for any $V, D \in \Div(R)$,
%	$$\phi(D)(V) = V.$$
%	
%	If $S$ is another Dedekind prime ring, and $\fracideal{R}{K}{S}$ is a fractional ideal, we have $O_r(K) = S$, and therefore, using Theorem \ref{general-mainthm},
%	$$V \circ D = V \otimes K + D .$$
%	In fact, tensoring by $K$ gives a canonical identification between the simple modules of $R$ and $S$ (it is independent of the choice of the fractional ideal), so composition $\circ$ is essentially just addition. 
%	%The theory of divisors therefore extends the very straightforward theory in Dedekind prime rings. 
%\end{mainexample}

We also point out a connection between divisors and ideals, which essentially states that divisors can be described by a collection of fractional ideals, and describes how multiplication of divisors works using these ideals.

\begin{theorem} \label{mult-mainthm}
	Let $R$ and $S$ be bounded HNP rings, let $V \in \Div(R)$ and $D \in \Div(S)$.
	Let $\overline{S} \supseteq S$ be any finite overring. Then, there is a finite subring $R_0 \subseteq R$, a finite overring $T \supseteq \overline{S}$, and a fractional ideal $\fracideal{R_0}{I}{T}$ such that
	$$D \otimes T = \p{I_T} $$
	and we have
	$$ (V \circ D) \otimes T = V \otimes_{R_0} I + D \otimes T \text{.}$$
\end{theorem}

%What follows is dedicated to proving the theorems stated in this section.

We will prove Theorems \ref{general-mainthm} and \ref{mult-mainthm} later in this section. Before proceeding we need to establish a somewhat general result about connected HNP rings.
Lemma \ref{inv-ideal} showed that if there is an invertible ideal between two bounded HNP rings, then we have a bijection between the simple modules, which preserves the cycle structure. This is no longer true if HNP rings are not connected by an invertible fractional ideal (e.g.\ take any HNP ring an its finite overring), but nevertheless, we still have a bijection between the cycles themselves.

Namely, let $R$ and $S$ be two connected bounded HNP rings. We denote by $\T_R$ the set of all cycles of simple $R$-modules. If we take a cycle $C \in \T_R$, and a fractional ideal $\fracideal{R}{I}{S}$, since $O_l(I)$ is a finite extension of $R$, there is a simple modules $U \in C$ such that $U \otimes I \neq 0$. In fact, we claim that the divisor $U \otimes I \in \Div(S)$ consists of distinct simple $S$-modules $V_i$ with
\begin{align} \label{V-arrows}
	V_n \rightarrow \dots \rightarrow V_1 \rightarrow V_0.
\end{align}
Indeed, $U \otimes O_l(I)$ is a simple $O_l(I)$-module. Since $I$ is an invertible $(O_l(I), O_r(I))$-ideal, Lemma \ref{inv-ideal} shows that the module $U \otimes I = U \otimes O_l(I) \otimes I$ is a simple $O_r(I)$-module. Since $S$ is a finite HNP subring of $O_r(I)$, applying \cite[Theorem 20.2]{LR}
%and LR Theorem 6.4
(and considering Remark \ref{remark}) proves the claim.

In particular, modules $V_i$ all belong to the same cycle $C_{U, I} \in \T_S$; this is the cycle, corresponding to $C$. Apart from the choice of $C$, the cycle $C_{U, I}$ seems to depends both on the choice of the fractional ideal $I$ and the simple module $U$, however, the following proposition shows that this actually not the case, i.e., the cycle $C_{U, I}$ only depends on $C$. Therefore, this gives us a well-defined map from $\T_R$ to $\T_S$, which is easily seen to be a bijection (tensor with the ideal $I^{-1}$ to obtain the map in the other direction).

\begin{proposition} \label{ideal-asoc}
		For any two bounded connected HNP rings $R$ and $S$, there is a canonical bijection between their cycles of simple modules, i.e., the cycle $C_{ U, I}$ described above does not depend on the choice of $U$ (from the original cycle of $R$-modules) or the ideal $I$.
\end{proposition}

\begin{proof}
	First, we will show that the cycle $C_{U, I}$ does not depend on the module $U_R$.
	 Let $W_R$ be another simple $R$-module with $W \otimes I \neq 0$, that belongs to the cycle of $U$. We claim that the simple modules from the Jordan-Hölder decomposition of $W \otimes I$ belong to the same cycle as the simple modules from the Jordan-Hölder decomposition $U \otimes I$.
	Since $I$ is an invertible $(O_l(I),O_r(I))$-ideal, by Lemma \ref{inv-ideal}, the simple $O_r(I)$-modules $U \otimes I$ and $W \otimes I$ belong to the same cycle of simple $O_r(I)$-modules. Applying \cite[Theorem 20.2]{LR} proves the claim. We will therefore write $C_I$ for $C_{U, I}$.

	Second, we will show that $C_I$ does not depend on $I$.
	Let $\fracideal{R}{J}{S}$ be another fractional ideal, and denote by $C_J$ the cycle of $S$-modules associated to $C$ by the ideal $J$. Since \cite[Theorem 20.2]{LR} shows that the Jordan-Hölder decomposition of
	$$ (U \otimes J \otimes J^{-1})_R = (U \otimes O_l(J))_R$$
	consists only of simple modules in the cycle $C$, the ideal $J^{-1}$ associates to $C_J$ the cycle $C$. Therefore, the ideal $K = J^{-1} I $ associates the cycle $C_J$ to the cycle $C_I$. Let $V$ be a simple $S$-module in the cycle $C_J$ with $V \otimes O_l(K) \neq 0$. Then, by Theorem \ref{star-mult}, we have
	$$ V \otimes K = \phi(\p{K})(V) ,$$
	and by Definition \ref{def} it is clear that $V \otimes K$ only consists of the simple modules in the cycle $C_J$. We conclude that $C_I = C_J$.
\end{proof}

\begin{lemma} \label{R0-exp}
	Let $R$ and $T$ be bounded HNP rings, and assume that there is a fractional ideal $\fracideal{R}{K}{T}$ with $T = O_r(K)$, and that every simple module appearing in $\p{K}_T$ has a trivial cycle. Let $U_1, \ldots, U_t$ be simple $T$-modules with trivial cycles. Then, there is a finite HNP subring $R_0 \subseteq R$ such that for any divisor
	$$D \in \bigoplus_{i=1}^t \Z \cdot U_i \leq \Div(T)$$
	there is a unique fractional ideal $\fracideal{R_0}{I}{T}$ with $ D = \p{I}$. 
\end{lemma}

\begin{proof}
	Denote $S= O_l(K)$. Let $V_1, \ldots, V_k$ be all simple $T$-modules appearing in $\p{K}$.
	By Lemma \ref{inv-ideal}, the $S$-modules $U_i \otimes K^{-1}$ and $V_j \otimes K^{-1}$ have trivial cycles. %for $i=1, \ldots, t$ and $j =1,\ldots, k$.
	%
%	Let $V_S$ be a simple $S$-module with a trivial cycle. We claim that $V \otimes K^{-1}$ is a simple $\overline{R}$-module with a trivial cycle. To observe simplicity, take $N_S \leq S$ to be a maximal one-sided ideal with
%	$$ V_S \cong S/N \text{.}$$
%	Then, $V \otimes K^{-1} \cong K^{-1} / NK^{-1}$ is also a simple $\overline{R}$-module since $K$ is invertible as an $(\overline{R},S)$-ideal (multiplying by $K$ preserves strict inclusions). Since $M = \ann V \lhd S$ is maximal, the ideal $K M K^{-1} \lhd T$ is also maximal and it annihilates $V \otimes K^{-1}$ (since $K^{-1} K M K^{-1} \leq N K^{-1}$). Therefore, $V \otimes K^{-1}$ is a simple module with a trivial cycle, which proves the claim.
	%
%	In particular, $$U_i \otimes \overline{R} = (U_i \otimes K) \otimes K^{-1}$$
%	are simple modules with a trivial cycle.
%	
%	Let $V_1, \ldots, V_k$ be all simple $S$-modules appearing in $\p{K}$. Since $V_i$ have trivial cycles by assumption, by the previous claim, the simple $\overline{R}$-modules $V_i \otimes K^{-1}$ also have trivial cycles.
%	Let $U'_i$ (resp.\ $V'_j$) be any simple $R$-modules with $U'_i \otimes \overline{R} = U_i \otimes K^{-1}$ (resp.\ $V'_i \otimes \overline{R} = V_i \otimes K^{-1}$ ).
	%
%	The ring $\overline{R}$ is admissible w.r.t. $U_1, \ldots, U_n$ and $U'_1, \ldots, U'_k$.
	%
	Let
	$$E = D \otimes K^{-1} + \p{K^{-1}} \in \Div(S) \text{.}$$
	By Proposition \ref{mult-tens}, we have $\p{K^{-1}} = - \p{K} \otimes K^{-1}$, and thus, only $U_i \otimes K^{-1}$ and $V_j \otimes K^{-1}$ appear in the above expression. Let $R_0 \subseteq R$ be the subring defined just as in the proof of Theorem \ref{star-mult}. By Lemma \ref{ideal-lemma}, there is an ideal $\fracideal{R_0}{J}{S}$ such that
	$$ E = \p{J} \text{.}$$
	Then, for $I = J K$, by Proposition \ref{mult-tens}, we have
	$$
		\p{I} = \p{J} \otimes K + \p{K} = D + \p{K^{-1}} \otimes K + \p{K}  = D \text{.}
	$$
	
	Now, assume that there is another ideal $\fracideal{R_0}{I'}{T}$ with
	$$ D = \p{I'} \text{.}$$
	Then, by Proposition \ref{mult-tens}, it follows that
	$$ E = \p{I' K^{-1}} \text{,}$$
	and by uniqueness in Lemma \ref{ideal-lemma}, we have
	$$ J = I' K ^{-1} \text{.}$$
	Multiplying the above equality by $K$, we obtain $I = I'$.
	\end{proof}
	
	The following lemma shows that a ring $T$ from the previous proposition always exists.

\begin{lemma} \label{trans-ideal-exists} 
	Let $R$ be a bounded HNP ring, $S$ an HNP ring connected to $R$, and let $U_1, \ldots, U_n$ be simple $S$-modules contained in different cycles. Then, there is a finite overring $T \supseteq S$ and a fractional ideal $\fracideal{R}{K}{T}$ with $O_r(K) = T$ such that the simple modules appearing in $\p{K}_T$ have trivial cycles, and $U_i \otimes T$ are simple $T$-modules with trivial cycles.
\end{lemma}

%\begin{lemma} %\label{trans-ideal-exists} 
%	Let $R$ be a bounded HNP ring, $S$ an HNP ring connected to $R$, and $U_1, \ldots, U_n$ simple $S$-modules contained in different cycles. Then there is a fractional ideal $\fracideal{R}{K}{S}$ such that for $T= O_r(K)$, the simple modules appearing in $\p{K}_T$ have trivial cycles, and $U_i \otimes T$ are simple $T$-modules with trivial cycles.
%\end{lemma}

\begin{proof}
	Let $\fracideal{R}{I}{S}$ be any fractional ideal, and let $\overline{S}$ be a finite overring of $S$ that trivializes all the cycles of $U_i$, but doesn't kill $U_i$ for all $i$, and also trivializes all the cycles of simple modules appearing in $\p{I}$. Let $K = I \overline{S}$, and $T = O_r(K)$. Since $K = I T$, by Proposition \ref{mult-tens}, we have
	$$\p{K} = \p{I} \otimes T.$$
	Since $\overline{S}$ trivializes all the cycles of simple modules appearing in $\p{I}$, every cycle of simple $T$-modules appearing in $\p{K}$ is trivial. Similarly, modules $U_i \otimes T$ are simple with trivial cycles.
\end{proof}

The following lemma shows that knowing some “local” information about ideal is enough to determine this “local” information about the tensor product of any divisor with that ideal.

\begin{lemma} \label{same-mult}
	Let $R$ and $S$ be bounded HNP rings, $V \in \Div(R)$ a divisor, and $\fracideal{R}{I}{S}$ and $\fracideal{R}{J}{S}$ fractional ideals. If the divisors $\p{I_S}$ and $\p{J_S}$ have the same multiplicity at a simple $S$-module $U$, then
	$$ V \otimes I \quad \text{and} \quad V \otimes J$$
	have the same multiplicity at $U$ as well.
\end{lemma}

\begin{proof}
	Let $\fracideal{R}{K}{S}$ be any fractional ideal such that $U \otimes O_r(K)$ is a simple module (which exists by Lemma \ref{trans-ideal-exists}). By assumption, the divisors $\p{I} \otimes O_r(K)$ and $\p{J} \otimes O_r(K)$ have the same multiplicity at $U \otimes O_r(K)$. Since
	$$ \p{I K^{-1}} = \p{I} \otimes K^{-1} + \p{K^{-1}} ,$$
	and similar for $J$, and tensoring by $K^{-1}$ is a bijection between simple $O_r(K)$-modules and simple $O_l(K)$-modules (by Lemma \ref{inv-ideal}), the divisors $\p{IK^{-1}}$ and $\p{JK^{-1}}$ have the same multiplicity at the simple $O_l(K)$-module $U \otimes K^{-1}$. By Theorem \ref{star-mult} and Definition \ref{def},
	$$ V \otimes IK^{-1} \quad \text{and} \quad V \otimes JK^{-1}$$
	have the same multiplicity at $U \otimes K^{-1}$. Tensoring by $K$, shows that
	$$ (V \otimes I ) \otimes O_r(K) \quad \text{and} \quad (V \otimes J ) \otimes O_r(K)$$
	have the same multiplicity at $U \otimes O_r(K)$, which implies the desired conclusion.
\end{proof}

By inspecting associativity of $\circ$ (a property that we want to hold, but did not prove yet), it is clear how $\circ$ should be defined using $\star$ (just as in Theorem \ref{general-mainthm}). However, it is less obvious if $\circ$ can be defined in a consistent fashion. The following theorem shows that this can indeed be done.

\begin{theorem} \label{circ-def-thm}
	For bounded connected HNP rings $R$ and $S$, and divisors $V \in \Div(R)$ and $D \in \Div(S)$, there is precisely one divisor
	$$ E \in \Div(S)$$
	such that for any ideal $\fracideal{R}{K}{S}$, we have
	$$ E \otimes {O_r(K)} = \big( V \star (D \otimes K^{-1} + \p{K^{-1}}) \big) \otimes K + \p{K} \in \Div(O_r(K)).$$
\end{theorem}

\begin{proof}%[Proof of Thorem \ref{circ-def-thm}.]
	For a simple $S$-module $U$, let $T \supseteq S$ be an overring described in Lemma \ref{trans-ideal-exists} (for $U_1 = U$), let $\fracideal{R_0}{I}{T}$ be an ideal with $D \otimes T = \p{I}_T$ (which exists by Lemma \ref{R0-exp}), and define $k_{U,T}$ to be the coefficient of the simple module $U \otimes T$ in the divisor
	\begin{align} \label{E-def} V \otimes I + D \otimes T \in \Div(T) .\end{align}
	We claim that $k_{U,T}$ depends only on $U$ and not on a choice of $T$; thus, we will denote it by $k_U$. Indeed, let $T'$ be another appropriate overring of $S$ and $\fracideal{R'_0}{J}{T'}$ a fractional ideal with $D \otimes T' = \p{J_{T'}}$. By tensoring \eqref{E-def} with $T'$, we see that $k_{U, T}$ is also the coefficient of
	$$ V \otimes IT' + (D \otimes T) \otimes T' \in \Div(T') $$
	at $U \otimes T' = U \otimes T \otimes T'$ (considering the discussion next to \eqref{V-arrows}). Since the divisors $D \otimes T'$ and $D \otimes T \otimes T'$ have the same coefficient at $U \otimes T'$, in light of Lemma \ref{same-mult} it is enough to show that
	$$ \p{IT'} = \p{I_S} \otimes T' $$
	and $\p{J_{T'}} = D \otimes T'$ have the same coefficient at $U \otimes T'$. By the discussion next to \eqref{V-arrows}, the divisors $\p{I_S}$ and $D$ have the same coefficient at $U$, and therefore, so do the divisors $\p{I_S} \otimes T'$ and $D \otimes T'$ at $U \otimes T'$, which proves the claim.

	Note that only finitely many of $k_U$ are nonzero. Indeed, if $U_S$ is a simple module that does not share a cycle with some simple module from $D$ or is not in a cycle that is associated to some cycle that involves a module from $V$ via Lemma \ref{ideal-asoc}, then $k_U = 0$. We can see this from the equality \eqref{E-def}.
	
	We define the divisor $E \in \Div(S)$ by setting
	$$ E = \sum_{U \in \M_S} k_U \cdot U.$$
	It remains to show that for an arbitrary $\fracideal{R}{K}{S}$, we have
	$$  E \otimes O_r(K) = \big( V \star (D \otimes K^{-1} + \p{K^{-1}}) \big) \otimes K + \p{K},$$
	from where uniqueness of $E$ is also apparent.
	To see this,  let $S' \supseteq O_r(K)$ be any right finite overring trivializing the cycles of all the simple modules appearing in $E \otimes O_r(K)$ and
	$$\big( V \star (D \otimes K^{-1} + \p{K^{-1}}) \big) \otimes K + \p{K} .$$
	Let $T \supseteq S'$ be an overring that trivializes the cycles in all the simple modules appearing in $D$ and $\p{K}$. Note that $T$ satisfies the assumptions of Lemma \ref{R0-exp} (w.r.t.\ ideal $L = KT$), and therefore, there is a fractional ideal $\fracideal{R_0}{I}{T}$ such that $D \otimes T = \p{I_T}$.
	We have
	\begin{align} \label{K-bar}
		\begin{split}
			\Big( \big( V \star (D \otimes K^{-1} &+ \p{K^{-1}}) \big) \otimes K +  \p{K} \Big) \otimes T \\ &=  \big( V \star (D \otimes L^{-1} + \p{L^{-1}}) \big) \otimes L +  \p{L} .
		\end{split}
	\end{align}
	Indeed, to see this use Proposition \ref{mult-tens}, the fact that $L = O_l(L) \otimes L$, Lemma \ref{tensor-lemma}, and finally, the equality $K^{-1} O_l(L) =  L^{-1}$. Now, since
	$$D \otimes L^{-1} = D \otimes T \otimes L^{-1} = \p{I} \otimes L^{-1},$$
	using Proposition \ref{mult-tens} and Theorem \ref{star-mult}, we obtain
	\begin{align} \label{K-bar2}
		\begin{split}
			\big( V \star (D \otimes L^{-1} &+ \p{L^{-1}}) \big) \otimes L +  \p{L} = V \otimes I + D \otimes T .
		\end{split}
	\end{align}
	Let $W$ be a simple $T$-module that appears in either $E \otimes T$ or \eqref{K-bar}, and let $U$ be a simple $S$-module with $U \otimes T = W$. The cycle of $W$ is then trivial, and by definition of $E$, the integer $k_U$ is the coefficient of the divisor
	$$V \otimes I + D \otimes T$$
	at module $W$. By considering \eqref{K-bar2}, it is clear that the divisors $E \otimes T$ and \eqref{K-bar} have the same coefficient at $W$. Since $W$ was arbitrary, we conclude that
	$$ E \otimes T = \Big( \big( V \star (D \otimes K^{-1} + \p{K^{-1}}) \big) \otimes K +  \p{K} \Big) \otimes T .$$
	Applying Lemma \ref{extensions} concludes the proof.
\end{proof}

\begin{definition} \label{circ-def}
	Let $R$ and $S$ be bounded HNP rings, let $V \in \Div(R)$ and $D \in \Div(S)$. We define
	$ V \circ D$
	to be the divisor $E \in \Div(S)$ described in Theorem \ref{circ-def-thm}.
\end{definition}

It is clear that Theorem \ref{general-mainthm} holds (for the operation $\circ$ given in Definition \ref{circ-def}).
In order to prove Theorem \ref{mult-mainthm}, we will first show that $\circ$ behaves well when tensoring with overrings (similar to Lemma \ref{tensor-lemma}).

\begin{lemma} \label{overring-lemma}
	Let $R$ and $S$ be bounded HNP rings, let $V \in \Div(R)$ and $D \in \Div(S)$. For any finite overring $T \supseteq S$, we have
	$$ (V \circ D) \otimes T = V \circ (D \otimes T) \text{.}$$
\end{lemma}

\begin{proof}
	Let $\overline{T} \supseteq T$ be a right finite overring that trivializes the cycles of all simple modules that appear in $(V \circ D) \otimes T$ and $V \circ (D \otimes T)$.
	Let $\fracideal{R}{K}{\overline{T}}$ be any fractional ideal.
	We clearly have
	\begin{align*}
		\big( (V \circ D) \otimes T \big)  \otimes O_r(K) &= (V \circ D) \otimes O_r(K) \\
		&= \big( V \star (D \otimes K^{-1} + \p{K^{-1}})  \big) \otimes K + \p{K} ,
	\end{align*}
	where in the second line we used Theorem \ref{circ-def-thm} for  $E = V \circ D$.
	On the other hand, by Theorem \ref{circ-def-thm} for $E = V \circ (D \otimes \overline{T})$, we also have
	\begin{align*}
		\big( V \circ (D \otimes T) \big) \otimes O_r(K) &= \big( V \star ((D \otimes T) \otimes K^{-1}+ \p{K^{-1}}) \big) \otimes K + \p{K} \\
		&= \big( V \star (D \otimes K^{-1} + \p{K^{-1}})  \big) \otimes K + \p{K} .
	\end{align*}
	Applying Lemma \ref{extensions} concludes the proof.
\end{proof}

%The following theorem extends Theorem \ref{star-mult}.

%\begin{theorem} \label{circ-ideal-mult-thm}
\begin{proof}[Proof of Theorem \ref{mult-mainthm}]
	Let $R$ and $S$ be bounded HNP rings. In light of Lemma \ref{R0-exp} and Lemma \ref{trans-ideal-exists}, it is enough to show that for a finite subring $R_0 \subseteq R$, a divisor $V \in \Div(R)$, and a fractional ideal $\fracideal{R_0}{I}{S}$, we have
	$$ V \circ \p{I_S} = V \otimes I +  \p{I} \text{.}$$
%\end{theorem}

%\begin{proof}
	Let $T \supseteq S$ be a right finite overring that trivializes the cycles of all simple modules that appear in $V \circ \p{I}$ and $\p{I} + V \otimes I$. Let $\fracideal{R}{K}{T}$ be any fractional ideal.
%	By Lemma \ref{overring-lemma} and \eqref{mult-tens}, we have
%	$$ (V \circ \p{I}) \otimes O_r(K) = V \circ \p{I O_r(K)} \text{,}$$
	By Theorem \ref{circ-def-thm} and Proposition \ref{mult-tens}, we have
	\begin{align*}
		(V \circ \p{I}) \otimes O_r(K)  &= \big( V \star (\p{I} \otimes K^{-1} + \p{K^{-1}}) \big) \otimes K + \p{K} \\
		&= (V \star \p{IK^{-1}}) \otimes K + \p{K} ,
	\end{align*}
	and hence, by Theorem \ref{star-mult} and Proposition \ref{mult-tens},
	\begin{align*}
		(V \star \p{IK^{-1}}) \otimes K + \p{K} &= (V \otimes IK^{-1} + \p{IK^{-1}} ) \otimes K + \p{K}\\
		&= V \otimes I O_r(K) + \p{I O_r(K)}  \\
		&=(V \otimes I + \p{I})\otimes O_r(K) .
	\end{align*}
	Applying Lemma \ref{extensions} concludes the proof.
\end{proof}

%%%%%%%%%%%%%%%%%%%%%%%%%%%%%%%%%%%%%%%%%%%%%%%
\section{Category, Functoriality, and Uniqueness} \label{sec-cat}

We will now define the category of divisors $\Div$. Set objects of $\Div$ to be connected HNP rings (in $Q$, a fixed ring of quotients) and the morphisms from $R$ to $S$ to be the divisors $\Div(R)$. For a morphism $D \in \Div(S)$ from $S$ to $R$, and a morphism $V \in \Div(R)$ from $R$ to $T$, we define their composition to be $V \circ D$ given in Definition \ref{circ-def}. We claim that this is indeed a category. By definition it is clear that
$$ 0_R \circ D = D \text{.}$$
Since $\p{S} = 0_S$, Theorem \ref{mult-mainthm} shows that
$$ V \circ 0_S = V \otimes S,$$
which is just $V$ if $S = R$. This shows that left and right unit laws hold. The following proposition shows that $\circ$ is in fact associative, and with this $\Div$ is a category.

\begin{proposition}
	Let $R$, $S$, and $T$ be bounded connected HNP rings, and let $V \in \Div(R)$, $D \in \Div(S)$, and $E \in \Div(T)$. Then,
	$$ (V \circ D) \circ E = V \circ (D \circ E ) \text{.}$$
	%i.e., the divisors form a category.
\end{proposition}

\begin{proof}
	%	Without loss of generality we may assume that $V$ is a simple module and thus, there is an ideal $J_R$ such that $V = \p{J}$.
	Let $T'$ be an overring of $T$ that trivializes all the cycles of simple modules appearing in both $(V \circ D) \circ E$ and $V \circ (D \circ E )$.
	By Lemma \ref{R0-exp}, there is a right finite overring $T'' \supseteq T'$ and a subring $S_0 \subseteq S$ such that there is an ideal $\prescript{}{S_0}{I}_{T''}$ with $E \otimes T'' = \p{I}$.
	
	Using Lemma \ref{R0-exp} again, we obtain $\overline{S} \supseteq O_l(I)$ and $R_0 \subseteq R$ and an ideal $\fracideal{R_0}{J}{\overline{S}}$ such that $D_{\overline{S}} = \p{J}$, denote $\overline{I} = \overline{S} I$. The ring $\overline{T} = O_r(\overline{I}) =  I^{-1} \overline{S} I$ is a right finite overring of $T''$ and we have $E \otimes {\overline{T}} = \p{\overline{I}}$.
	Using Lemma \ref{overring-lemma} and Theorem \ref{mult-mainthm}, we obtain
	\begin{align*}
		\big( (V \circ D) \circ E \big) \otimes \overline{T} &= (V \circ D) \circ \p{\overline{I}} \\
		&= (V \circ D) \otimes \overline{I} + \p{\overline{I}} \\
		&= (V \circ D) \otimes \overline{S} \otimes \overline{I} + \p{\overline{I}} \\
		&= (V \otimes J + \p{J} ) \otimes \overline{I} + \p{\overline{I}} \\
		&= V \otimes J \overline{I} + \p{J} \otimes \overline{I} + \p{\overline{I}}
	\end{align*}
	and
	\begin{align*}
		\big( V \circ (D \circ E ) \big) \otimes \overline{T} &= V \circ \big( (D \circ E) \otimes \overline{T} \big) \\
		&= V \circ (D \circ \p{\overline{I}})\\
		&= V \circ (D \otimes \overline{I} + \p{\overline{I}}) \\
		&= V \circ (D \otimes \overline{S} \otimes \overline{I} + \p{\overline{I}})\\
		&= V \circ (\p{J} \otimes \overline{I} + \p{\overline{I}}) \\
		&= V \circ \p{J \overline{I}} \text{.}
	\end{align*}
	Since the above divisors are equal by Theorem \ref{mult-mainthm}, applying Lemma \ref{extensions} concludes the proof.
\end{proof}

We have already defined a map
$$\p{} \colon \Ical_R(S) \rightarrow \Div(S),$$
assigning to a fractional ideal (a morphism in the category of ideals) a divisor (a morphism in the just defined category of divisors). For $I \in \Ical_R(S)$ and $J \in \Ical_T(R)$, Theorem \ref{mult-mainthm} shows that
$$ \p{J} \circ \p{I} = \p{J} \otimes I + \p{J}.$$
Proposition \ref{mult-tens} then implies that $\p{}$ is in fact a functor.

By Theorem \ref{mult-mainthm}, if $R_0 \subseteq R$, we have
$$V \circ 0_{R_0} = V \otimes_{R_0} R_0 = V_{R_0} .$$
To show that the operation $\circ$ satisfies conditions of Theorem \ref{uniquness-mainthm}, we still have to shows that $\_ \circ D - D \colon \Div(R) \rightarrow \Div(S)$ is an additive map; this is done in the following proposition.

\begin{proposition}
	Let $R$ and $S$ be bounded HNP rings, and let $U,V \in \Div(R)$ and $D \in \Div(S)$. Then, we have
	$$ (U + V) \circ D + D = U \circ D + V \circ D. $$
\end{proposition}

\begin{proof}
	Let $S'$ be an overring of $S$ that trivializes all the cycles of simple modules appearing in both $(U + V) \circ D + D $ and $U \circ D + V \circ D$.
	By Lemma \ref{R0-exp}, there is a right finite overring $T \supseteq S'$ and a subring $R_0 \subseteq R$ such that there is an ideal $\prescript{}{R_0}{I}_T$ with $D \otimes T= \p{I}$.
	%	Denote by $\overline{S} = O_r(I)$.
	%
	By Lemma \ref{overring-lemma} and Theorem \ref{circ-def-thm}, we have
	\begin{align*}
		\big(  E \circ D \big) \otimes T&=  E \circ (D \otimes T)\\
		&= E \circ \p{I}  \\
		&= E \otimes I + \p{I} .
	\end{align*}
	for any divisor $E \in \Div(R)$. Using this, we can easily see that
	$$ \big( (U + V) \circ D + D \big) \otimes T = \big( U \circ D + V \circ D \big) \otimes T.$$
	Applying Lemma \ref{extensions} concludes the proof.
\end{proof}

\begin{proof}[Proof of Theorem \ref{uniquness-mainthm}]
	In view of what we previously discussed in this section we only need to shows uniqueness  of the operation $\circ$. Let $\diamond$ be another composition of divisors satisfying the assumptions in this theorem, i.e.,
	$$\p{JI} = \p{J} \diamond \p{I},$$
	 the map $V \mapsto V \diamond D - D$ is additive, and $V \diamond 0_{R_0}= V_{R_0}$ for $R_0 \subseteq R$.
	 We claim that for any $V \in \Div(R)$ and $\fracideal{R_0}{I}{S}$ with $R_0 \subseteq R$ a finite subring, we have
	$$ V \diamond \p{I} = V \otimes I + \p{I},$$
	and in particular, if $I = S$ is an overring of $R$,
	\begin{align} \label{diamond-overring} V \diamond 0_S = V \otimes S .\end{align}
	%	Then, since $R$ is an $(R, R)$-fractional ideal, we have
	%	$$ E \diamond 0_{R} = \p{J_R} \diamond \p{R_{R}} = \p{JR_{R}} = E .$$
	Indeed, since $V \mapsto V \diamond D-D$ is additive, we have
	$$ 0_{R_0} \diamond D = D,$$
	and hence,
	\begin{align*}
		V \diamond \p{I} &= V \diamond (0_{R_0} \diamond \p{I})
		= (V \diamond 0_{R_0}) \diamond \p{I} 
		= V_{R_0} \diamond \p{I}.
	\end{align*}
	We can write $V = \sum_{i=1}^n \p{J_i}$ for some fractional right $R$-ideals $J_i$. Since $V \mapsto V \diamond \p{I} - \p{I}$ is additive, we have
	\begin{align*}
		V_{R_0} \diamond \p{I} &= \big( \sum_{i=1}^n \p{J_i} \big) \diamond \p{I} \\
		&= \sum_{i=1}^n \p{J_i} \diamond \p{I} - (n-1) \, \p{I} \\
		&= \sum_{i=1}^n \p{J_i I} - (n-1) \, \p{I}.
	\end{align*}
	Using Proposition \ref{mult-tens} yields
	\begin{align*}
		V_{R_0} \diamond \p{I} &= \sum_{i=1}^n (\p{J_i} \otimes I + \p{I}) - (n-1) \, \p{I}\\
		&= \sum_{i=1}^n \p{J_i} \otimes I  + \p{I} \\
		&= V \otimes I + \p{I},
	\end{align*}
	which proves the claim.
	
	Let $\overline{S} \supseteq S$ be a right finite overring that trivializes the cycles of all simple modules that appear in $V \circ D$ and $V \diamond D$. By Lemma \ref{R0-exp}, there is a finite subring $R_0 \subseteq R$, a finite overring $T\supseteq \overline{S}$, and an ideal $\fracideal{R_0}{I}{T}$ with $D \otimes T= \p{I_T}$. By \eqref{diamond-overring} and associativity, we have
	\begin{align*}
	(V \diamond D) \otimes T &= (V \diamond D) \diamond 0_{T} \\
	&= V \diamond (D \diamond 0_{T}) \\
	&= V \diamond (D \otimes T) .
	\end{align*}
	and since $D \otimes T = \p{I_T}$, we have
	\begin{align*}
	(V \diamond D) \otimes T &= V \diamond \p{I_T} \\
	&= V \otimes I + D \otimes T .
	\end{align*}
	By Lemma \ref{overring-lemma} and Theorem \ref{mult-mainthm}, we also have
	$$(V \circ D) \otimes T = V \otimes I + D \otimes T .$$
	Applying Lemma \ref{extensions} shows that
	$$ V \diamond D = V \circ D . \eqno{\qedhere}$$
\end{proof}

\bibliographystyle{alphaabbr}
\bibliography{references}

\end{document}